\documentclass[11pt,reqno]{amsart}

\usepackage{amsfonts, amsthm, amsmath}
\usepackage{amssymb}
\usepackage{graphicx}
\usepackage{amscd}
\usepackage{color}
\usepackage{float}
\usepackage{graphics,amsmath,amssymb}
\usepackage{amsthm}
\usepackage{amsfonts}
\usepackage{latexsym}
\usepackage{epsf}
\usepackage{xifthen}
\usepackage{mathrsfs}
\usepackage{dsfont}
\usepackage{makecell}
\usepackage[FIGTOPCAP]{subfigure}
\usepackage[justification=centering,singlelinecheck=false]{caption}
\usepackage{amsmath}
\allowdisplaybreaks[4]
\usepackage{listings}
\usepackage{etoolbox}
\usepackage{fancyhdr}
\usepackage{pdflscape}
\usepackage[title,toc,titletoc]{appendix}
\usepackage{enumitem}
\usepackage[noadjust]{cite}
\usepackage{tikz}
\usepackage{hyperref}

\allowdisplaybreaks[4]

\numberwithin{equation}{section}

\theoremstyle{theorem}
\newtheorem{theorem}{Theorem}[section]
\newtheorem*{theorem*}{Theorem}

\newtheorem{corollary}[theorem]{Corollary}
\newtheorem{lemma}[theorem]{Lemma}

\newtheorem{innercustomgeneric}{\customgenericname}
\providecommand{\customgenericname}{}
\newcommand{\newcustomtheorem}[2]{%
	\newenvironment{#1}[1]
	{%
		\renewcommand\customgenericname{#2}%
		\renewcommand\theinnercustomgeneric{##1}%
		\innercustomgeneric
	}
	{\endinnercustomgeneric}
}
\newcustomtheorem{ctheorem}{Theorem}
\newcustomtheorem{clemma}{Lemma}

\theoremstyle{definition}
\newtheorem{definition}{Definition}[section]

\newtheorem*{example*}{Example}
\newtheorem*{examples*}{Examples}
\newtheorem{remark}{Remark}[section]
\newtheorem*{remark*}{Remark}
\newtheorem*{remarks*}{Remarks}

\newtheoremstyle{named}{}{}{\itshape}{}{\bfseries}{.}{.5em}{#1\thmnote{ #3}}
\theoremstyle{named}

\newcommand\qbin[2]{{\begin{bmatrix} #1 \\ #2 \end{bmatrix} }}

\newcommand{\Res}{\operatorname{Res}}

\newcommand{\qbinom}[2]{\begin{bmatrix}#1\\#2\end{bmatrix}}

\newcommand{\LHS}{\operatorname{LHS}}
\newcommand{\RHS}{\operatorname{RHS}}

\newcommand{\Aux}{\operatorname{Aux}}

\newcommand{\fS}{\mathfrak{S}}

\newcommand{\uk}{\underline{k}}

\newcommand{\trans}{\mathsf{T}}

\newcommand{\alt}{\mathrm{a}}

\begin{document}
	
\title[]
{The Ariki--Koike algebras and Rogers--Ramanujan type partitions}

\author[S. Chern]{Shane Chern}
\address[S. Chern]{Department of Mathematics and Statistics, Dalhousie University, Halifax, Nova Scotia, B3H 4R2, Canada}
\email{xh375529@dal.ca; chenxiaohang92@gmail.com}

\author[Z. Li]{Zhitai Li}
\address[Z. Li]{Department of Mathematics, The Pennsylvania State University, University Park, PA 16802, USA}
\email{zfl5082@psu.edu}

\author[D. Stanton]{Dennis Stanton}
\address[D. Stanton]{School of Mathematics, University of Minnesota, Minneapolis, MN 55455, USA}
\email{stant001@umn.edu}

\author[T. Xue]{Ting Xue}
\address[T. Xue]{School of Mathematics and Statistics, University of Melbourne, Parkville, Victoria 3010, Australia}
\email{ting.xue@unimelb.edu.au}

\author[A. J. Yee]{Ae Ja Yee}
\address[A. J. Yee]{Department of Mathematics, The Pennsylvania State University, University Park, PA 16802, USA}
\email{yee@psu.edu}





\dedicatory{}

\maketitle

\begin{abstract}
In 2000,  Ariki and Mathas showed that the simple modules of the Ariki--Koike algebras $\mathcal{H}_{\mathbb{C},q;Q_1,\ldots, Q_m}\big(G(m, 1, n)\big)$ (when the parameters are roots of unity and $q\neq 1$) are labeled by the so-called Kleshchev multipartitions.  This together with Ariki's categorification theorem enabled Ariki and Mathas to obtain the generating function for the number of Kleshchev multipartitions by making use of the Weyl--Kac character formula. 
In this paper, we revisit this generating function for the $q=-1$ case. This $q=-1$ case is particularly interesting, for the corresponding Kleshchev multipartitions  have  a very close connection to generalized Rogers--Ramanujan type partitions when $Q_1=\cdots=Q_a=-1$ and $Q_{a+1}=\cdots =Q_m =1$. Based on this connection, we  provide an analytic proof of the result of Ariki and Mathas for $q=Q_1=\cdots Q_a=-1$ and $Q_{a+1}=\cdots =Q_m =1$.  
Our second objective is to investigate simple modules of the Ariki--Koike algebra in a fixed block.   It is known that these simple modules in a fixed block are labeled by the Kleshchev multiparitions with a fixed partition residue statistic.  This partition statistic is also studied  in the works of Berkovich, Garvan, and Uncu. Employing their results, we provide two bivariate generating function identities when $m=2$.
\end{abstract}
		
\section{Introduction}

The most fascinating identities in the theory of partitions are the Rogers--Ramanujan identities, which were originally proved by Rogers \cite{Rogers1984} and rediscovered by Ramanujan \cite{Ram1919-1}.  These identities have been a great source of research in the past several decades. They were reproved through many different approaches in the literature, and there also exist numerous similar identities and generalizations. Such identities are called Rogers--Ramanujan type identities or generalized Rogers--Ramanujan identities.  

The Rogers--Ramanujan identities appear in other areas as well, for instance, in  the representation theory of Lie algebra \cite{Lepowsky, Milne},
proving their mathematical relevance. In this paper, we will discuss another deep connection of these identities to the representations of the so-called Ariki--Koike algebras, or cyclotomic Hecke algebras of  type $G(m,1,n)$. 

The Ariki--Koike algebras, denoted by $\mathcal{H}_{\mathbb{C},q;Q_1,\ldots, Q_m}\big(G(m, 1, n)\big)$, can be viewed as the Iwahori--Hecke algebras associated to the complex reflection groups $G(m,1,n)\cong S_n\ltimes (\mathbb{Z}/m\mathbb{Z})^n$, where $q,Q_i,i=1,\ldots,m$ are parameters. They were introduced by Ariki and Koike \cite{ArikiKoike} and independently by Brou\'e and Malle \cite{BroueMalle}, where Brou\'e and Malle defined cyclotomic Hecke algebras for all complex reflection groups. In \cite{Ariki96,AM}, Ariki and Mathas showed that the simple modules of the Ariki--Koike algebras (when the parameters are roots of unity) are labeled by the so-called Kleshchev multipartitions. These multipartitions are in general defined recursively and no simple description is known except when $q=-1$. In the latter case, Mathas in~\cite{mathas} gave a simple combinatorial description of these multipartitions. 

In this paper, we are interested in the set $\Lambda^{a,m}(n)$ of multipartitions, whose exact definition will be given in Section~\ref{sec:prel} using Mathas' description, which parametrizes the simple modules of 
$$
\mathcal{H}_{\mathbb{C},q;Q_1,\ldots, Q_m}\big(G(m, 1, n)\big),
$$
where $Q_1=\cdots =Q_{a}=-1$, $Q_{a+1}=\cdots =Q_m=1$, and $q=-1$. 
  
The generating function for $\Lambda^{a,m}(n)$ can be deduced as a special case of a theorem due to Ariki and Mathas.
\begin{theorem}[Ariki--Mathas \cite{AM}]\label{thm:AM}
We have
\begin{align}
\sum_{n\ge 0} |\Lambda^{a,m}(n)| x^n 
&= \prod_{n\ge 1} \frac{(1-x^{(m+2)n})(1-x^{(m+2)(n-1) +(a+1)} )(1-x^{(m+2)n - (a+1)})}{(1-x^n)(1-x^{2n-1})}.
\end{align}
\end{theorem}

This theorem is particularly intriguing, for the infinite product representation of the generating function is well-known to the community of the partition theory. This infinite product without the factor $(1-x^{2n-1})$ in the denominator appears in the well-known generalized Rogers--Ramanujan identities of Andrews, Bressoud, and Gordon \cite{And1966, Bress, Gordon}. Even with that factor, this infinite product itself appears in Andrews' pioneering work on parity questions in classical partition identities \cite{And2010} and a follow-up paper of Kim and Yee \cite{KY2013}. 

In this paper, we will revisit Theorem~\ref{thm:AM} from a partition-theoretic perspective by showing that $\Lambda^{a,m}(n)$ is equinumerous with the set of the generalized Rogers--Ramanujan partitions with gap conditions studied in \cite{And2010} and \cite{KY2013}. In particular, we will be led to new analytic identities of Andrews--Gordon type.

Our second objective is to study the set $\Lambda^{a,m}(n)$ with partition residue statistics.   Partitions possess rich arithmetic properties, one of which is the well-known Ramanujan's partition congruences \cite{Ram1919}. In the study of such properties, combinatorial statistics play an important role. For instance, in \cite{GKS}, Garvan, Kim and Stanton introduced the so-called crank statistic to prove Ramanujan's congruences combinatorially.  Their crank statistic stems from core and quotient partitions, which are fundamental concepts in the representation theory of symmetric groups. 
An essential concept in defining core and quotient partitions is residue statistics. 

The mod $2$ residue (or $2$-residue) statistic for partitions was also studied by Berkovich and Garvan. In their paper on the Andrews--Stanley srank for integer partitions \cite{BerkovichGarvan2006}, they  defined a new partition statistic, namely, the BG-rank, and they showed that the BG-rank gives another combinatorial account for Ramanujan's mod $5$ partition congruence. Here, the BG-rank is in essence the $2$-residue statistic \cite{BerkovichGarvan2008}.  Berkovich and Uncu investigated further obtaining a refined generating function for partitions with a fixed BC-rank \cite{berkovich}.

In this paper, we will consider this $2$-residue statistic for multipartitions in $\Lambda^{a,2}(n)$ to obtain the following bivariate generating function identities. For a multipartition $\mu$ in $\Lambda^{a,2}(n)$,  we denote by $\omega(\mu)$ the $2$-residue of $\mu$, whose definition will be given in Section~\ref{sec:prel}.

\begin{theorem}\label{thm:1.2}
We have
\begin{align}
\sum_{n\ge 0} \sum_{\mu\in \Lambda^{1,2}(n)} y^{\omega(\mu)} x^{n} =\prod_{n\ge 1} \frac{(1+x^{2n})}{(1-x^{2n})} \prod_{n\ge 1} (1+ yx^{2n-1}) (1+y^{-1} x^{2n-1}) (1-x^{2n}) , \label{eq:1.2}
\end{align}
and
\begin{align}
& \sum_{n\ge 0} \sum_{\mu\in \Lambda^{2,2}(n)} y^{\omega(\mu)} x^{n}  \notag \\
& = \frac{1}{2} \prod_{n\ge 1}\frac{(1+x^{2n-1})}{(1- x^{2n})} \prod_{n\ge 1}   \Big(  (1+yx^{2n-2})(1+y^{-1}x^{2n}) + (1-yx^{2n-2})(1-y^{-1}x^{2n})\Big) (1- x^{2n}).  \label{eq:1.3}
\end{align}
\end{theorem}

Let us write $\Lambda_\omega^{a,2}(n)$ for the set of multipartitions in $\Lambda^{a,2}(n)$ with 2-residue $\omega$. An application of Jacobi's triple product identity to each of the right hand sides of \eqref{eq:1.2} and \eqref{eq:1.3} yields the following corollary. 
\begin{corollary}\label{coro:1.4}
We have
\begin{equation}\label{number1}
\sum_{n\geq 0}|\Lambda_\omega^{1,2}(n)|x^n=x^{\omega^2}\prod_{n\ge 1}  \frac{ 1+x^{2n} }{1-x^{2n} },
\end{equation}
and
\begin{equation}\label{number2}
\qquad \sum_{n\geq 0}|\Lambda_\omega^{2,2}(n)|x^n=x^{\omega^2-\omega} \prod_{n\ge 1} \frac{ 1+x^{2n-1} }{ 1-x^{2n} }.
\end{equation}
\end{corollary}

By~\cite{LM}, it is known that the set $\Lambda_{\omega}^{a,2}(n)$  parametrizes the set of simple modules in a fixed block, labeled by $\omega$, of the Ariki--Koike algebra $\mathcal{H}_{-1;-1, 1}\big(G(2, 1, n)\big)$ (resp.~$\mathcal{H}_{-1;-1, -1}\big(G(2, 1, n)\big)$), when $a=1$ (resp.~$2$). Note that in this case the Ariki--Koike algebra is an Iwahori--Hecke algebra of a Weyl group of type $B_n$ (or $C_n$). 
It is interesting to note that $|\Lambda_0^{1,2}(2n)|=|\Lambda_1^{1,2}(2n+1)|$ (resp.~$|\Lambda_0^{2,2}(2n)|$, $|\Lambda_1^{2,2}(2n+1)|$), the number of simple modules in the principal block of the corresponding Hecke algebra, equals the number of nilpotent orbits in the Lie algebra of type $C_n$ (resp.~$D_n$, $B_n$). 

 The rest of this paper is organized as follows. In Section~\ref{sec:prel}, we recall some definitions, notation, and necessary results on basic hypergeometric series.  In Section~\ref{sec:2-restricted},  we present some combinatorial facts about partitions in $\Lambda^{a,2}(n)$, and in Section~\ref{sec:bivariate}, we prove Theorem~\ref{thm:1.2}.  In Section~\ref{sec5}, we give an analytic proof of Theorem~\ref{thm:AM}.  We then conclude our paper providing some remarks in Section~\ref{sec6}. 


\section{Preliminaries} \label{sec:prel}
In this section, we recall basic partition definitions and the $q$-Pochhammer symbol notation. We also collect some necessary results on basic hypergeometric series for later use. 

\subsection{Partition definitions}

A \textit{partition} of $n$ is an integer sequence $\lambda=(\lambda_1,\ldots, \lambda_\ell)$ such that  $\lambda_1\ge \cdots \ge \lambda_{\ell}>0$ and $|\lambda|:=\lambda_1+\cdots +\lambda_\ell=n.$ We write it as $\lambda \vdash n$. The $\lambda_i$ are parts of $\lambda$ and the number of parts of $\lambda$ is denoted by $\ell(\lambda)$.   

For $\lambda \vdash n$, its {\it Young diagram}, also known as the {\it Ferrers diagram}, is the graphical representation, which consists of $n$ boxes (or dots) placed in rows such that there are $\lambda_i$ boxes (or dots) in the $i$-th row.  We denote the Young diagram of $\lambda$ by $Y_{\lambda}$. The {\it conjugate partition} of $\lambda$, denoted by $\lambda^{\trans}$, is the partition associated to the Young diagram resulting from reflecting the Young diagram of $\lambda$ about the main diagonal. In Figure~\ref{fig0},  the Young diagram of $\lambda=(5,4,4,2)$ is illustrated, and $\lambda^{\trans}=(4,4,3,3,1)$.

\begin{figure}[ht]
    \bigskip
    \centering
	\begin{tikzpicture}[scale=1.0]
		\draw[black] (2, 1) -- (4.5, 1);
		\draw[black] (2, 0.5) -- (4.5, 0.5);
		\draw[black] (2, 0) -- (4, 0);
		\draw[black] (2, -0.5) -- (4, -0.5);
		\draw[black] (2, -1) -- (3, -1);
		
		\draw[black] (2, 1) -- (2, -1);
		\draw[black] (2.5, 1) -- (2.5, -1);
		\draw[black] (3, 1) -- (3, -1);
		\draw[black] (3.5, 1) -- (3.5, -0.5);
		\draw[black] (4, 1) -- (4, -0.5);
		\draw[black] (4.5, 1) -- (4.5, 0.5);
		
			\end{tikzpicture}
	  \caption{ $Y_{(5, 4, 4, 2)}$}
    \label{fig0}
\end{figure}
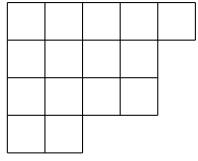 

For a positive integer $p$, the {\it residue} of a node $(i,j) \in Y_{\lambda}$ mod $p$ is defined to be
$$
\text{Res}(i,j):= (j-i) \,\bmod{p}. 
$$
Figure~\ref{fig00} shows the residues mod $3$ of $\lambda=(5,4,4,2)$. 
\begin{figure}[ht]
    \bigskip
    \centering
	\begin{tikzpicture}[scale=1.0]
		\draw[black] (2, 1) -- (4.5, 1);
		\draw[black] (2, 0.5) -- (4.5, 0.5);
		\draw[black] (2, 0) -- (4, 0);
		\draw[black] (2, -0.5) -- (4, -0.5);
		\draw[black] (2, -1) -- (3, -1);
		
		\draw[black] (2, 1) -- (2, -1);
		\draw[black] (2.5, 1) -- (2.5, -1);
		\draw[black] (3, 1) -- (3, -1);
		\draw[black] (3.5, 1) -- (3.5, -0.5);
		\draw[black] (4, 1) -- (4, -0.5);
		\draw[black] (4.5, 1) -- (4.5, 0.5);
				
		\draw (2.25,0.7) node{\footnotesize{$0$}};
		\draw (2.25,0.2) node{\footnotesize{$2$}};
		\draw (2.25,-0.3) node{\footnotesize{$1$}};
		\draw (2.25,-0.8) node{\footnotesize{$0$}};
		
		\draw (2.75,0.7) node{\footnotesize{$1$}};
		\draw (2.75,0.2) node{\footnotesize{$0$}};
		\draw (2.75,-0.3) node{\footnotesize{$2$}};
		\draw (2.75,-0.8) node{\footnotesize{$1$}};
		
		\draw (3.25,0.7) node{\footnotesize{$2$}};
		\draw (3.25,0.2) node{\footnotesize{$1$}};
		\draw (3.25,-0.3) node{\footnotesize{$0$}};

		\draw (3.75,0.7) node{\footnotesize{$0$}};
		\draw (3.75,0.2) node{\footnotesize{$2$}};
		\draw (3.75,-0.3) node{\footnotesize{$1$}};
		
		\draw (4.25,0.7) node{\footnotesize{$1$}};
			\end{tikzpicture}
	  \caption{ Residues mod $3$}
    \label{fig00}
\end{figure} 

For a positive integer $e$, a partition $\lambda$ is {\it $e$-restricted} if 
$$ 
\lambda_{i}-\lambda_{i+1}<e \text{ for $1\le i\le \ell(\lambda)$},
$$ 
where $\lambda_{\ell(\lambda)+1}=0$. 
In other words, if $\lambda$ is $e$-restricted, then its conjugate $\lambda^{\trans}$ can have at most $e-1$ parts of the same size.  
 
We now introduce multipartitions. An $m$-multipartition of $n$ is an $m$-tuple of partitions $(\lambda^{(1)},\ldots, \lambda^{(m)})$ such that
$$
 \sum_{i=1}^m |\lambda^{(i)}|=n.
$$
By abuse of notation, we denote a multipartition as $\lambda=(\lambda^{(1)},\ldots, \lambda^{(m)})$ and 
$$
|\lambda|:=\sum_{i=1}^m |\lambda^{(i)}|.
$$
Also, if $|\lambda|=n$, then we say $\lambda$ is a multipartition of $n$, and we write $\lambda \vdash n$. 

\begin{definition}
For a positive integer $e$ and a sequence ${\bf t}=(t_1,\ldots, t_m)$, an $m$-multipartition $\lambda$ is ($e, {\bf t})$-restricted if
\begin{itemize} 
\item[(i)] each $\lambda^{(i)}$ is $e$-restricted;
\item[(ii)] $\ell(\lambda^{(i)}) +t_i \le \lambda_{1}^{(i+1)} + t_{i+1}$  for $1\le i \le m-1 $.
\end{itemize}
\end{definition}

\begin{definition}
For an $(e,{\bf t})$-restricted partition $\lambda=(\lambda^{(1)}, \ldots, \lambda^{(m)})$, 
let 
$$
 Y_{\lambda} :=\{(i,j,s)\, |\, (i,j) \in Y_{\lambda^{(s)}}, \, s=1,\ldots, m\},
$$
where $Y_{\lambda^{(s)}}$ denotes the Young diagram of $\lambda^{(s)}$. Define the {\it $p$-residue} of a node $x=(i,j,s)\in Y_{\lambda^{(s)}}$ for $s=1,\ldots, m$ by
$$
\Res(x):=(j-i +t_s)\, \bmod{p}.  
$$

\end{definition}

\subsection{Notation}
Throughout the rest of this paper, we will adopt the following \textit{$q$-Pochhammer symbols}:
\begin{equation*}
(a;q)_{\infty}:=\prod_{j\ge 0} (1-aq^{j}),
\end{equation*}
and
\begin{equation*}
    (a;q)_n:=\frac{(a;q)_{\infty}}{(aq^n;q)_{\infty}} \text{ for any integer $n$.}
\end{equation*}
{
Here we note that this parameter $q$ is independent of the parameter $q$ for the Ariki--Koike algebras $\mathcal{H}_{\mathbb{C},q;Q_1,\ldots, Q_m}\big(G(m, 1, n)\big)$.}

Also, for brevity, the following notation will be used frequently:
\begin{align*}
	(a_1, a_2, \ldots, a_r;q)_n &:= (a_1;q)_n(a_2;q)_n \cdots (a_r;q)_n,\\
	(a_1, a_2, \ldots, a_r;q)_{\infty} &:= (a_1;q)_\infty (a_2;q)_\infty \cdots (a_r;q)_\infty.
\end{align*}

The \textit{$q$-binomial coefficients}, also known as the \textit{Gaussian polynomials}, are given by
\begin{align*}
	\qbinom{N}{M}:=\qbinom{N}{M}_q:=\begin{cases}
		\scalebox{0.85}{%
			$\dfrac{(q;q)_N}{(q;q)_M(q;q)_{N-M}}$} 
			& \text{if $0\le M\le N$},\\[6pt]
		0 & \text{otherwise}.
	\end{cases}
\end{align*}
They satisfy two basic recurrences \cite[p.~35, eqs.~(3.3.4) and (3.3.3)]{And1976}:
\begin{align}
	\qbinom{N}{M}&=\qbinom{N-1}{M} {\color{black} + } q^{N-M}\qbinom{N-1}{M-1}, \label{binomtri1}\\
	\qbinom{N}{M}&=\qbinom{N-1}{M-1}  {\color{black} + } q^{M}\qbinom{N-1}{M}. \label{binomtri2}
\end{align}
Also, the following  trivial relation follows from their definition:
\begin{align}\label{eq:q-binomial-trivial}
	(1-q^{M})\qbinom{N}{M}=(1-q^N)\qbinom{N-1}{M-1}.
\end{align}

\subsection{Some lemmas on basic hypergeometric series}

Recall that the \textit{basic hypergeometric series }   ${}_{r+1}\phi_r$ is defined by
\begin{align*}
	{}_{r+1}\phi_{r} \left(\begin{matrix} A_1,A_2,\ldots,A_{r+1}\\ B_1,B_2,\ldots,B_r  \end{matrix}; q, z\right):=\sum_{n\ge 0} \frac{(A_1,A_2,\ldots,A_{r+1};q)_n \, z^n}{(q,B_1,B_2,\ldots,B_{r};q)_n}.
\end{align*}

We list a set of useful identities: 
\begin{itemize}[leftmargin=*,align=left]
\renewcommand{\labelitemi}{$\triangleright$}
\item 
The $q$-binomial theorem \cite[(3.3.6)]{And1976}:
\begin{align}\label{eq:q-Bin}
	(z;q)_N = \sum_{n\ge 0} (-1)^n z^n q^{\binom{n}{2}}\qbinom{N}{n}_q.
\end{align}

\item The $q$-binomial theorem \cite[(II.3)]{GR2004}:
\begin{align}\label{eq:q-Bin-Ser}
	\sum_{n\ge 0}\frac{(a;q)_n z^n}{(q;q)_n} = \frac{(az;q)_\infty}{(z;q)_\infty}.
\end{align}

\item Jacobi's triple product identity \cite[(II.28)]{GR2004}:
\begin{align}\label{eq:JTP}
	(q,z,q/z;q)_\infty = \sum_{n=-\infty}^\infty (-1)^n z^n q^{\binom{n}{2}}.
\end{align}

\item Heine's first transformation \cite[(III.1)]{GR2004}:
\begin{align}\label{eq:Heine-1}
	{}_{2}\phi_{1}\left(\begin{matrix} a,b\\ c \end{matrix}; q, z\right) = \frac{(b,az;q)_\infty}{(c,z;q)_\infty} {}_{2}\phi_{1}\left(\begin{matrix} c/b,z\\ az \end{matrix}; q, b\right).
\end{align}

\item Heine's second transformation \cite[(III.2)]{GR2004}:
\begin{align}\label{eq:Heine-2}
	{}_{2}\phi_{1}\left(\begin{matrix} a,b\\ c \end{matrix}; q, z\right) = \frac{(c/b,bz;q)_\infty}{(c,z;q)_\infty} {}_{2}\phi_{1}\left(\begin{matrix} abz/c,b\\ bz \end{matrix}; q, \frac{c}{b}\right).
\end{align}
\end{itemize}

\bigskip

We also establish the following evaluation of a well-poised ${}_{2}\phi_1$ series.
\begin{lemma}
	We have
	\begin{align}\label{71+}
		{}_{2}\phi_1 \left(\begin{matrix} a,b \\ aq^2/b \end{matrix} ;  q^2, q/b\right) = \frac{(q;q^2)_{\infty} \big( (-\sqrt{a}, q\sqrt{a}/b ; q)_{\infty} + (\sqrt{a}, -q\sqrt{a}/b; q)_{\infty} \big)}{2 (aq^2/b;q^2)_{\infty} (q/b;q^2)_{\infty}} .
	\end{align}
\end{lemma}

\begin{proof}
	Notice that
	\begin{align*}
		&{}_{2}\phi_1 \left(\begin{matrix} a,b \\ aq^2/b \end{matrix} ;  q^2, q/b\right) \\
		\text{\tiny (by \eqref{eq:Heine-1})}& =\frac{(a;q^2)_{\infty}(q;q^2)_{\infty}}{(aq^2/b;q^2)_{\infty} (q/b;q^2)_{\infty}} {}_{2}\phi_1 \left(\begin{matrix} q^2/b,q/b  \\ q\end{matrix} ;  q^2, a \right) \\
		&= \frac{(a;q^2)_{\infty}(q;q^2)_{\infty}}{(aq^2/b;q^2)_{\infty} (q/b;q^2)_{\infty}}  \sum_{k=0}^{\infty} \frac{(q/b;q)_{2k}}{(q;q)_{2k}} a^k\\
		&= \frac{(a;q^2)_{\infty}(q;q^2)_{\infty}}{2 (aq^2/b;q^2)_{\infty} (q/b;q^2)_{\infty}}  \left( \sum_{k=0}^{\infty} \frac{(q/b;q)_{k}}{(q;q)_{k}} \sqrt{a}^k +\sum_{k=0}^{\infty} \frac{(q/b;q)_{k}}{(q;q)_{k}} (-\sqrt{a})^k\right) \\
		\text{\tiny (by \eqref{eq:q-Bin-Ser})}&= \frac{(a;q^2)_{\infty}(q;q^2)_{\infty}}{2 (aq^2/b;q^2)_{\infty} (q/b;q^2)_{\infty}} \left(\frac{(q\sqrt{a}/b;q)_{\infty}}{(\sqrt{a};q)_{\infty}} + \frac{(-q\sqrt{a}/b;q)_{\infty}}{(-\sqrt{a};q)_{\infty}} \right) \\
		&= \frac{(q;q^2)_{\infty}  \big( (-\sqrt{a}, q\sqrt{a}/b;q)_{\infty} + (\sqrt{a}, -q\sqrt{a}/b  ;q)_{\infty} \big)  }{2 (aq^2/b;q^2)_{\infty} (q/b;q^2)_{\infty}}.
	\end{align*}
	This is exactly \eqref{71+}.
\end{proof}


\section{$2$-Restricted multipartitions} \label{sec:2-restricted}

For $m\ge 2$ and $1\le a\le m$,  we set
$$
{\bf e }= (2, {\bf t}),
$$
where
\begin{equation}\label{eq:t-def}
	t_1=\cdots =t_a=0, \quad t_{a+1}=\cdots =t_m=1.
\end{equation}

It easily follows from the definition of restricted multipartitions that 
a multipartition is ${\bf e }$-restricted if and only if its conjugate $\pi=(\pi^{(1)}, \ldots, \pi^{(m)})$ is a multipartition such that
\begin{enumerate}
\item[(i)] each $\pi^{(i)}$ is a strict partition, i.e., parts are all distinct;
	
\item[(ii)] $\pi_1^{(i)}  +t_i\le \ell(\pi^{(i+1)}) +t_{i+1}$ for $1\le i\le m-1$.
\end{enumerate}
Throughout this section and the rest of this paper, we will use this equivalent definition for ${\bf e}$-restricted partitions. 

Let 
$$
\Lambda^{a,m}:= \{ \text{ $(2,{\bf t})$-restricted  $m$-multipartitions }   \}
$$ 
and
$$
\Lambda^{a,m}(n):=\{ \pi\in \Lambda^{a,m}\, |\, \pi \vdash n\}. 
$$

Recall the residues of multipartitions from Section~\ref{sec:prel}.

For $\pi \in \Lambda^{a,m}$, define
$$
\omega(\pi):=C_0(\pi) -C_1(\pi),
$$
where
$$
C_i (\pi) := | \{ x\in Y_{\pi} \, |\, \text{Res} (x) = i \}|.
$$
For a fixed integer $\omega$, we also define
$$
\Lambda_{\omega}^{a,m}:=\{ \pi \in \Lambda^{a,m}  \,|\, \omega(\pi)=\omega \},
$$
and
$$
\Lambda_{\omega}^{a,m}(n):=\{\pi \in  \Lambda^{a,m}(n) \, |\, \omega(\pi)=\omega\}. 
$$ 
Note that $\omega(\pi)\equiv |\pi| \bmod{2}$, so  if $\omega \not\equiv n  \bmod{2}$, then $\Lambda_{\omega}^{a,m}(n)$ is the empty set.

\subsection{2-Multipartitions}

Our objective is to investigate how mod $2$ residues behave for partition pairs in $\Lambda^{a,2}$ for $a=1,2$.  We will first consider the $a=2$ case and then the $a=1$ case. But, to bring out a delicate account of the behavior of the residues, we need a few more definitions and results from \cite{berkovich}, which will be discussed in the next subsection.

\subsubsection{Odd and even-indexed parts}

For a partition $\lambda$, we call $\lambda_1,\, \lambda_3,\, \lambda_5,\, \ldots$ odd-indexed parts, and $\lambda_2,\, \lambda_4,\, \lambda_6,\, \ldots$ even-indexed parts.  Let $d_N(i,j,n)$ count the number of partitions of $n$ into distinct parts less than or equal to $N$ with $i$ odd-indexed odd parts and $j$ even-indexed odd parts. For instance,  the following partition is counted by $d_{10}(2, 3, 43)$: 
$$
(10, 9, 7, 6, 5, 3, 2, 1). 
$$

Define
$$
B_{N}(k,q):=\sum_{j\ge 0} \sum_{n\ge 0} d_{N} (j+k, j, n) q^n. 
$$
\begin{theorem}[\cite{berkovich}, Theorem 4.1] \label{berkovich2}
	For $v=0,1$,
	\begin{align*}
		B_{2N+v}(k,q) = q^{2k^2-k}  \begin{bmatrix} 2N +v \\ N+k \end{bmatrix}_{q^2}.
	\end{align*}
\end{theorem}

Let
\begin{align*}
	f_{N}(x,q):=\sum_{k=-\infty}^{\infty} B_{N}(k,q) x^{k}.
\end{align*}

For a partition $\lambda$, let 
$$
|\lambda|_{\alt}:=\sum_{i\ge 1} (-1)^{i-1} \lambda_i,
$$ 
which is called the alternating sum of $\lambda$. Let
$$
g_N(x,q):=\sum_{\lambda}  x^{|\lambda|_{\alt}} q^{|\lambda|},
$$
where the sum is over all partitions $\lambda$ with parts less than or equal to $N$.  We have the following nice formula for $g_N(x,q)$.

\begin{theorem}[{\cite[Theorem 5.2]{berkovich}}, \cite{zeng}] \label{zeng}
	For $v=0,1$, 
	\begin{align*}
		g_{2N+v}(x,q)&=\frac{1}{(q^2;q^2)_{2N+v}} 
		\sum_{i=0}^N \begin{bmatrix} N\\ i \end{bmatrix}_{q^{4}} x^{2i} q^{2i} (-xq;q^4)_{N-i+v}(-x^{-1}q; q^4)_{i}\\
		&=\frac{1}{(q^2;q^2)_{2N+v}} \sum_{j=0}^{2N+v} x^j q^j\begin{bmatrix} 2N+v\\ j\end{bmatrix}_{q^2}.
	\end{align*}
\end{theorem}

\begin{proof}
	The last equality follows from the identities for the Rogers--Szeg\H{o} polynomial $H_n(t,q)$:
	\begin{align*}
		H_n(t,q)=\sum_{j=0}^{n} t^j \begin{bmatrix} n\\ j\end{bmatrix} =\sum_{r=0}^{\lfloor n/2\rfloor} t^{2r} (-q/t;q^2)_r (-t;q^2)_{\lceil n/2 \rceil -r} \begin{bmatrix} \lfloor n/2 \rfloor \\ r \end{bmatrix}_{q^2}.
	\end{align*}
	See \cite[eq.~(8.3)]{BW2005}.
\end{proof}

\subsubsection{The $a=2$ case}\label{sec:a=2}

For a strict partition $\pi$, let
\begin{align*}
	\delta(\pi)&=(\ell(\pi), \ell(\pi)-1,\ldots, 1)\\
	\sigma(\pi)&=(\pi_1-\ell(\pi), \pi_2-\ell(\pi)+1,\ldots, \pi_{\ell(\pi)} -1)^{\trans}.
\end{align*}
Then, it is clear that
$$
\pi=\delta(\pi)+ \sigma(\pi)^{\trans}. 
$$
By the definition of residues with $t_1=t_2=0$, we see that
$$
\omega(\pi)=(-1)^{\ell(\pi)-1}\bigg( \bigg\lceil \frac{ \ell(\pi)}{2} \bigg\rceil - |\sigma(\pi)|_{\alt} \bigg).
$$
For example, we consider $\pi=(9,7,6,3)$. Then
\begin{align*}
	\delta(\pi)&=(4,3,2,1),\\
	\sigma(\pi)&=(5,4,4,2)^{\trans}=(4,4,3,3,1),
\end{align*}
and
\begin{align*}
	\omega(\pi) = 12-13=-1 &= (-1)^{4-1}\big(2 - (4-4+3-3+1)\big)\\
	&=(-1)^{\ell(\pi)-1} \bigg(\bigg\lceil \frac{ \ell(\pi)}{2} \bigg\rceil - |\sigma(\pi)|_{\alt}\bigg).
\end{align*}
In Figure~\ref{fig1}, the shifted Young diagram of $\pi$ with residues is given.  The thick vertical line separates $\delta(\pi)$ and $ \sigma(\pi)^{\trans}$. 
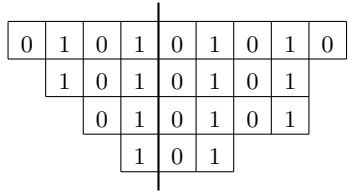
\begin{figure}[ht]
       \centering
	\begin{tikzpicture}[scale=1.0]
		\draw[black] (0, 1) -- (4.5, 1);
		\draw[black] (0, 0.5) -- (4.5, 0.5);
		\draw[black] (0.5, 0) -- (4, 0);
		\draw[black] (1, -0.5) -- (4, -0.5);
		\draw[black] (1.5, -1) -- (3, -1);
		
		\draw[black] (0, 1) -- (0, 0.5);
		\draw[black] (0.5, 1) -- (0.5, 0);
		\draw[black] (1, 1) -- (1, -0.5);
		\draw[black] (1.5, 1) -- (1.5, -1);
		\draw[black] (2, 1) -- (2, -1);
		\draw[black] (2.5, 1) -- (2.5, -1);
		\draw[black] (3, 1) -- (3, -1);
		\draw[black] (3.5, 1) -- (3.5, -0.5);
		\draw[black] (4, 1) -- (4, -0.5);
		\draw[black] (4.5, 1) -- (4.5, 0.5);
		
		\draw[thick,black] (2, 1.25) -- (2, -1.25);
		
		\draw (0.25,0.7) node{\footnotesize{$0$}};
		\draw (0.75,0.7) node{\footnotesize{$1$}};
		\draw (0.75,0.2) node{\footnotesize{$1$}};
		\draw (1.25,0.7) node{\footnotesize{$0$}};
		\draw (1.25,0.2) node{\footnotesize{$0$}};
		\draw (1.25,-0.3) node{\footnotesize{$0$}};
		\draw (1.75,0.7) node{\footnotesize{$1$}};
		\draw (1.75,0.2) node{\footnotesize{$1$}};
		\draw (1.75,-0.3) node{\footnotesize{$1$}};
		\draw (1.75,-0.8) node{\footnotesize{$1$}};
		
		\draw (2.25,0.7) node{\footnotesize{$0$}};
		\draw (2.25,0.2) node{\footnotesize{$0$}};
		\draw (2.25,-0.3) node{\footnotesize{$0$}};
		\draw (2.25,-0.8) node{\footnotesize{$0$}};
		
		\draw (2.75,0.7) node{\footnotesize{$1$}};
		\draw (2.75,0.2) node{\footnotesize{$1$}};
		\draw (2.75,-0.3) node{\footnotesize{$1$}};
		\draw (2.75,-0.8) node{\footnotesize{$1$}};
		
		\draw (3.25,0.7) node{\footnotesize{$0$}};
		\draw (3.25,0.2) node{\footnotesize{$0$}};
		\draw (3.25,-0.3) node{\footnotesize{$0$}};

		\draw (3.75,0.7) node{\footnotesize{$1$}};
		\draw (3.75,0.2) node{\footnotesize{$1$}};
		\draw (3.75,-0.3) node{\footnotesize{$1$}};
		
		\draw (4.25,0.7) node{\footnotesize{$0$}};
		
	\end{tikzpicture}
	 \caption{The strict partition $\pi=(9, 7, 6, 3)$}
    \label{fig1}
\end{figure} 

\noindent Hence, for $\pi=(\pi^{(1)}, \pi^{(2)})\in \Lambda^{2,2}$, 
$$
\omega(\pi)= \omega(\pi^{(1)})+\omega(\pi^{(2)})=\omega(\pi^{(1)})+  (-1)^{\ell(\pi^{(2)})-1} \bigg(\bigg\lceil \frac{ \ell(\pi^{(2)} )}{2} \bigg\rceil - |\sigma(\pi^{(2)})|_{\alt}\bigg).
$$
It follows that 
\begin{align}\label{eq:Lambda^{2,2}-gf-0}
	G_2(x,q)&:=\sum_{\omega=-\infty}^{\infty} \sum_{n \ge 0} |\Lambda_{\omega}^{2,2}(n )| \, x^{\omega} q^n \notag \\
	&=\sum_{(\pi^{(1)},\pi^{(2)})\in \Lambda^{2,2}} x^{\omega(\pi^{(1)})+\omega(\pi^{(2)})} q^{|\pi^{(1)}|+|\pi^{(2)}|}\notag\\
	&= \sum_{n\ge 0}  \sum_{(\sigma, \nu)} x^{(-1)^{n-1}\big(\lceil{n/2 \rceil - |\sigma|_{\alt}}\big)+\omega(\nu)} q^{n(n+1)/2+|\sigma|+|\nu|},
\end{align}
where  the inner sum in the last line is over all $(\sigma, \nu)$ with $\sigma$ an ordinary partition into parts $\le n$ and $\nu$ a strict partition into parts $\le n$. It turns out that this generating function $G_2(x,q)$ can be expressed as follows.

\begin{theorem}\label{th:Lambda^{2,2}-sum}
	We have
	\begin{align}
		&\sum_{\omega=-\infty}^{\infty} \sum_{n \ge 0} |\Lambda_{\omega}^{2,2}(n )| \, x^{\omega} q^n \notag\\
		&=\sum_{n\ge 0} \frac{x^{-n} q^{n(2n+1)}}{(q^2;q^2)_{2n}}  \sum_{j=0}^{2n} x^j q^j  \begin{bmatrix} 2n \\ j \end{bmatrix}_{q^2} \sum_{k=-n}^{n} x^{k} q^{2k^2-k}  \begin{bmatrix} 2n \\ n+k \end{bmatrix}_{q^2} \notag\\
		&\quad + \sum_{n\ge 0} \frac{x^{n+1} q^{(n+1)(2n+1)}}{(q^2;q^2)_{2n+1}} \sum_{j=0}^{2n+1} x^{-j} q^j \begin{bmatrix} 2n+1\\ j \end{bmatrix}_{q^2} \sum_{k=-n}^{n+1} x^k q^{2k^2-k} \begin{bmatrix} 2n +1 \\ n+k\end{bmatrix}_{q^2}.
	\end{align}
\end{theorem}

\begin{proof}
	Note that by \eqref{eq:Lambda^{2,2}-gf-0},
	\begin{align*}
		&\sum_{\omega=-\infty}^{\infty} \sum_{n \ge 0} |\Lambda_{\omega}^{2,2}(n)| x^{\omega} q^n \\
		&=\sum_{n\ge 0} x^{-n} q^{n(2n+1)} g_{2n}(x,q)  f_{2n}(x,q) +\sum_{n\ge 0} x^{n+1} q^{(n+1)(2n+1)}  g_{2n+1}(x^{-1},q) f_{2n+1}(x,q).
	\end{align*}
	Making use of Theorems \ref{berkovich2} and \ref{zeng} gives the desired result.
\end{proof}

We compute the coefficient of $x^{\omega}$: 
\begin{align}
	[x^{\omega}]G_2(x,q)
	&=\sum_{n\ge 0} \sum_{k=-\infty}^{\infty} \frac{q^{2(n^2+n+k^2-k) +\omega}}{(q^2;q^2)_{2n}} \begin{bmatrix} 2n\\ n-k +\omega \end{bmatrix}_{q^2}  \begin{bmatrix} 2n\\ n+k \end{bmatrix}_{q^2} \notag \\
	& \quad +\sum_{n\ge 0} \sum_{k=-\infty}^{\infty} \frac{q^{2((n+1)^2+k^2)-\omega}}{(q^2;q^2)_{2n+1}} \begin{bmatrix}  2n+1 \\ n-k+\omega \end{bmatrix} _{q^2} \begin{bmatrix} 2n+1 \\ n+k\end{bmatrix}_{q^2}. \label{g2omega}
\end{align}

\subsubsection{The $a=1$ case}\label{sec:a=1}

For $\pi=(\pi^{(1)}, \pi^{(2)})  \in \Lambda^{1,2} $, since ${\bf t}=(t_1, t_2)=(0,1)$, the analysis on the residue of $\pi^{(1)}$ from the previous section on $\Lambda^{2,2}$ is still valid; while the residue of each node of $\pi^{(2)}$ has changed by $1$ mod 2 due to $t_2=1$. Thus, for the weighted generating function for $\pi^{(2)}$, we have to replace the variable $x$ by $x^{-1}$. 
It follows that 
\begin{align}\label{eq:Lambda^{1,2}-gf-0}
	G_1(x,q)&:=\sum_{\omega=-\infty}^{\infty} \sum_{n \ge 0} |\Lambda_{\omega}^{1,2}(n)|\, x^{\omega} q^n \notag \\ 
	&=\sum_{(\pi^{(1)},\pi^{(2)})\in \Lambda^{1,2}} x^{\omega(\pi^{(1)})+\omega(\pi^{(2)})} q^{|\pi^{(1)}|+|\pi^{(2)}|}\notag\\
	&= \sum_{n\ge 0}  \sum_{(\sigma, \nu)} x^{(-1)^{n}\big(\lceil{n/2 \rceil - |\sigma|_{\alt}}\big)+ \omega(\nu)} q^{n(n+1)/2+|\sigma|+|\nu|},
\end{align}
where  the inner sum in the last line is over all $(\sigma, \nu)$ with $\sigma$ an ordinary partition into parts $\le n$ and $\nu$ a strict partition into parts $\le n+1$. Thus, this generating function $G_1(x,q)$ can be expressed as follows.

\begin{theorem}\label{th:Lambda^{1,2}-sum}
	We have
	\begin{align}
		&\sum_{\omega=-\infty}^{\infty} \sum_{n \ge 0} |\Lambda_{\omega}^{1,2}(n)| x^{\omega} q^n \notag\\
		&=\sum_{n\ge 0} \frac{x^{n} q^{n(2n+1)}}{(q^2;q^2)_{2n}}  \sum_{j=0}^{2n} x^{-j} q^j  \begin{bmatrix} 2n \\ j \end{bmatrix}_{q^2} \sum_{k=-n}^{n+1} x^{k} q^{2k^2-k}  \begin{bmatrix} 2n +1\\ n+k \end{bmatrix}_{q^2} \notag\\
		&\quad + \sum_{n\ge 0} \frac{x^{-(n+1)} q^{(n+1)(2n+1)}}{(q^2;q^2)_{2n+1}} \sum_{j=0}^{2n+1} x^{j} q^j \begin{bmatrix} 2n+1\\ j \end{bmatrix}_{q^2} \sum_{k=-n-1}^{n+1} x^{k} q^{2k^2-k} \begin{bmatrix} 2n +2 \\ n+1+k\end{bmatrix}_{q^2}.
	\end{align}
\end{theorem}

\begin{proof}
	Note that by \eqref{eq:Lambda^{1,2}-gf-0},
	\begin{align*}
		&\sum_{\omega=-\infty}^{\infty} \sum_{n \ge 0} |\Lambda_{\omega}^{1,2}(n )| x^{\omega} q^n \\
		&=\sum_{n\ge 0} x^{n} q^{n(2n+1)} g_{2n}(x^{-1},q)  f_{2n+1}(x,q) +\sum_{n\ge 0} x^{-(n+1)} q^{(n+1)(2n+1)}  g_{2n+1}(x,q) f_{2n+2}(x,q).
	\end{align*}
	Making use of Theorems \ref{berkovich2} and \ref{zeng} gives the desired result.
\end{proof}

We now compute the coefficient of $x^{\omega}$: 
\begin{align}
	[x^{\omega}]G_1(x,q)
	&=\sum_{n\ge 0} \sum_{k=-\infty}^{\infty} \frac{q^{2(n^2+n+k^2) +\omega}}{(q^2;q^2)_{2n}} \begin{bmatrix} 2n\\ n+k - \omega \end{bmatrix}_{q^2}  \begin{bmatrix} 2n+1\\ n+k \end{bmatrix}_{q^2} \notag \\
	& \quad  +\sum_{n\ge 0} \sum_{k=-\infty}^{\infty} \frac{q^{2((n+1)^2+k^2-k)-\omega}}{(q^2;q^2)_{2n+1}} \begin{bmatrix}  2n+1 \\ n+k-\omega \end{bmatrix} _{q^2} \begin{bmatrix} 2n+2 \\ n+1+k\end{bmatrix}_{q^2}. \label{g1omega}
\end{align}


\section{The bivariate generating functions for $\Lambda^{1,2}$ and $\Lambda^{2,2}$}\label{sec:bivariate}

In Theorems \ref{th:Lambda^{2,2}-sum} and \ref{th:Lambda^{1,2}-sum}, we have expressed the bivariate generating functions for $\Lambda^{1,2}$ and $\Lambda^{2,2}$ as triple $q$-summations.  In this section, we will simplify these triple summations and hence provide a proof of Theorem~\ref{thm:1.2}. 

For notational convenience, we rewrite the identities in the theorem using $q$-Pochhammer symbols as follows:
	\begin{align}\label{eq:Lambda12-bi-gf}
		\sum_{\omega=-\infty}^\infty \sum_{n\ge 0} |\Lambda_{\omega}^{1,2}(n)| x^{\omega}q^{n} = \frac{(-q^2;q^2)_{\infty}(-xq,-x^{-1}q,q^2;q^2)_\infty}{(q^2;q^2)_{\infty}},
	\end{align}
and
	\begin{align}\label{eq:Lambda22-bi-gf}
		\sum_{\omega=-\infty}^\infty \sum_{n\ge 0} |\Lambda_{\omega}^{2,2}(n)| x^{\omega}q^{n} &= \frac{(-q;q^2)_\infty(-x,-x^{-1}q^2,q^2;q^2)_\infty}{2(q^2;q^2)_\infty}+\frac{(q;q^2)_\infty(x,x^{-1}q^2,q^2;q^2)_\infty}{2(q^2;q^2)_\infty}.
	\end{align}

Our proof  relies on some basic hypergeometric series manipulations of the coefficients in \eqref{g2omega} and \eqref{g1omega}. Before we give the proof, we discuss some applications of these identities. 

First of all, they lead us to the following interesting basic hypergeometric series identities.

\begin{corollary}
	We have
	\begin{align*}
		\sum_{r,s\ge 0} \frac{q^{r^2+s^2+r+s}(q^2;q^2)_{r+s+1}}{(q^2;q^2)_r (q^2;q^2)_r (q^2;q^2)_s (q^2;q^2)_{s+1}}=\frac{(-q^2;q^2)_{\infty}}{(q^2;q^2)_{\infty}}.
	\end{align*}
\end{corollary}

\begin{proof}
	This result corresponds to the coefficient of $x^0$ in \eqref{eq:Lambda12-bi-gf}. Note that for the summation side, we may use \eqref{eq:G1-omega-rs}, and rewrite the $q$-binomial notation as $q$-Pochhammer symbols. For the product side, we apply Jacobi's triple product identity in \eqref{eq:JTP}. 
\end{proof}

\begin{corollary}
	We have
	\begin{align*}
		& \sum_{r,s\ge 0} \frac{q^{r^2+s^2+2s}(q^2;q^2)_{r+s}}{(q^2;q^2)_r(q^2;q^2)_s (q^2;q^2)_r (q^2;q^2)_s}
		+ \sum_{r,s\ge 0} \frac{q^{r^2+s^2+2r+2s+2}(q^2;q^2)_{r+s+1}}{(q^2;q^2)_{r}(q^2;q^2)_{r+1} (q^2;q^2)_{s}(q^2;q^2)_{s+1}}\\
		&\quad = \frac{(-q;q^2)_{\infty}}{(q^2;q^2)_{\infty}}.
	\end{align*}
\end{corollary}

\begin{proof}
	This result follows by summing the coefficients of $x^0$ and $x^1$ in \eqref{eq:Lambda22-bi-gf}. Note that for the summation side, we make use of \eqref{eq:G2-omega-rs}, and also rewrite the $q$-binomial coefficients as $q$-Pochhammer symbols. For the product side, we apply the Jacobi triple product identity in \eqref{eq:JTP}. 
\end{proof}

Furthermore, it immediately follows from \eqref{eq:Lambda12-bi-gf} and \eqref{eq:Lambda22-bi-gf} that
\begin{equation}
\sum_{n\ge 0} |\Lambda^{1,2} (n)| q^n = (-q;q)_{\infty} (-q;q^2)_{\infty} \label{g_1(1)}
\end{equation}
and
\begin{equation}
\sum_{n\ge 0} |\Lambda^{2,2}( n)| q^n = (-q;q)_{\infty} (-q^2;q^2)_{\infty}.  \label{q_2(1)}
\end{equation}

Also, since $\omega(\pi)\equiv n \pmod{2},$ if we sum over $\omega \equiv n \pmod{2}$, each of the generating function looks succinct.  
\begin{itemize} [leftmargin=*,align=left]
\renewcommand{\labelitemi}{$\triangleright$}
\item The $a=1$ case: 
\medskip
\begin{itemize}[leftmargin=*,align=left]
	\item[$\triangleright$] If $n \equiv 0 \bmod{2}$,  then
	\begin{align*}
		|\Lambda^{1,2}(n) |= \sum_{ \omega \equiv 0 \bmod{2}} | \Lambda^{1,2}_{\omega}(n)|.
	\end{align*}
	Hence,
	\begin{align*}
\sum_{\substack{n \ge 0 \\ n \equiv 0 \bmod{2}}}  | \Lambda^{1,2}(n)| q^n
&=\frac{1}{2}\left( (-q;q^2)^2_{\infty} (-q^2;q^2)_{\infty}+(q;q^2)^2_{\infty} (-q^2;q^2)_{\infty} \right) \\
&=\frac{1}{2} \left((-q;q^2)^2_{\infty} +(q;q^2)^2_{\infty}\right) (-q^2;q^2)_{\infty}\\
&=\frac{(-q^2;q^2)_{\infty} (-q^4; q^8)_{\infty}^2 (q^8;q^8)_{\infty}  }{(q^2;q^2)_{\infty}}.
\end{align*}

\item[$\triangleright$] If $n \equiv 1 \bmod{2}$, then 
\begin{align*}
|\Lambda^{1,2}(n) |= \sum_{ \omega \equiv 1 \bmod{2}} | \Lambda^{1,2}_{\omega}(n)|.
\end{align*}
Hence,
\begin{align*}
\sum_{\substack{ n \ge 0 \\ n \equiv 1 \bmod{2}}}  | \Lambda^{1,2}(n)| q^n
&=\frac{1}{2}\left( (-q;q^2)_{\infty}^2(-q^2;q^2)_{\infty}- (q;q^2)_{\infty}^2(-q^2;q^2)_{\infty} \right) \\
&=\frac{1}{2} \left((-q;q^2)^2_{\infty} - (q;q^2)^2_{\infty}\right) (-q^2;q^2)_{\infty}\\
&=\frac{2q (-q^2;q^2)_{\infty} (q^{16};q^{16})_{\infty}^2}{(q^2;q^2)_{\infty} (q^8;q^8)_{\infty}}.
\end{align*}
\end{itemize}

\begin{remark}
The last equality in each of the above generating functions can be derived as follows. We first note that
$$
(q;q^2)_{\infty}^2=\frac{(q;q)_\infty^2}{(q^2;q^2)_{\infty}^2}.
$$
Therefore, we may evaluate $(-q;q^2)_{\infty}^2 \pm (q;q^2)_{\infty}^2$ by recalling the 2-dissection formula of Ramanujan's classical theta function $\phi(-q)$ \cite[p. 49, Entry 25]{Ber1991}:
\begin{align*}
    \phi(-q):=\sum_{n=-\infty}^{\infty} (-1)^n q^{n^2} = \frac{(q;q)_\infty^2}{(q^2;q^2)_{\infty}} = \frac{(q^8;q^8)_{\infty}^5}{(q^4;q^4)^2_{\infty} (q^{16};q^{16})_{\infty}^2} - \frac{2q(q^{16};q^{16})_{\infty}^2}{(q^8;q^8)_{\infty}}.
\end{align*}
\end{remark}

\item The $a=2$ case:
\medskip
\begin{itemize}[leftmargin=*,align=left]
	\item[$\triangleright$] If $n \equiv 0 \bmod{2}$, then
	\begin{align*}
		|\Lambda^{2,2}(n) |= \sum_{ \omega \equiv 0 \bmod{2}} | \Lambda^{2,2}_{\omega}(n)|.
	\end{align*}
	Hence,
	\begin{align*}
	\sum_{\substack{n\ge 0 \\ n \equiv 0 \bmod{2}}}  | \Lambda^{2,2}(n)| q^n 
		&=\frac{1}{2}\left((-q;q^2)_{\infty} (-q^2;q^2)^2_{\infty}+(q;q^2)_{\infty} (-q^2;q^2)^2_{\infty} \right)\\
		&=\frac{1}{2}\left((-q;q^2)_{\infty} +(q;q^2)_{\infty}\right) (-q^2;q^2)_{\infty}^2 \\
		&=\frac{(-q^2;q^2)_{\infty}(-q^6,-q^{10}, q^{16};q^{16})_{\infty}}{(q^2;q^2)_{\infty}}.
	\end{align*}
	
	\item[$\triangleright$] If $n \equiv 1 \bmod{2}$,  then
	\begin{align*}
		|\Lambda^{2,2}(n) |= \sum_{\omega \equiv 1 \bmod{2}} | \Lambda^{2,2}_{\omega}(n)|.
	\end{align*}
	Hence,
	\begin{align*}
		\sum_{\substack{n \ge 0 \\ n \equiv 1 \bmod{2}}}  | \Lambda^{2,2}(n)| q^n 
		&=\frac{1}{2}\left((-q;q^2)_{\infty} (-q^2;q^2)^2_{\infty}- (q;q^2)_{\infty} (-q^2;q^2)^2_{\infty} \right)\\
		&=\frac{1}{2} \left((-q;q^2)_{\infty} - (q;q^2)_{\infty}\right) (-q^2;q^2)_{\infty}^2 \\
		&=\frac{q (-q^2;q^2)_{\infty}(-q^2,-q^{14}, q^{16};q^{16})_{\infty}}{(q^2;q^2)_{\infty}}.
	\end{align*}
\end{itemize}

\begin{remark}
	The last equality in each of the above generating functions can be derived as follows. First, note that
	\begin{align*}
		(q;q^2)_\infty = \frac{(q;q)_\infty}{(q^2;q^2)_\infty}.
	\end{align*}
	Therefore, we may evaluate $(-q;q^2)_{\infty} \pm (q;q^2)_{\infty}$ by recalling the 2-dissection formula of $(q;q)_\infty$ \cite[p.~332, (34.10.2)]{Hir2017}:
	\begin{align*}
		(q;q)_\infty &= (q^{12},q^{20};q^{32})_\infty (q^2,q^{14},q^{16};q^{16})_\infty\\
		&\quad - q (q^{4},q^{28};q^{32})_\infty (q^6,q^{10},q^{16};q^{16})_\infty.
	\end{align*}
\end{remark}
\end{itemize}

\subsection{Proof of \eqref{eq:Lambda12-bi-gf}}

Recall that
\begin{align*}
	G_1(x):=G_1(x,q)=\sum_{\omega=-\infty}^\infty \sum_{n\ge 0} |\Lambda_{\omega}^{1,2}(n)| x^{\omega}q^{n}.
\end{align*}
In \eqref{g1omega}, letting $r=n+k$ and $s=n-k$ in the first double summation and letting $r=n-k+1$ and $s=n+k$ in the second double summation, we have
\begin{align}\label{eq:G1-omega-rs}
	[x^{\omega}]G_1(x)
	&= \sum_{\substack{r,s=-\infty\\ r\equiv s \bmod 2}}^\infty \frac{q^{r^2+r+s^2+s-\omega}}{(q^2;q^2)_{r+s}}\qbinom{r+s}{r-\omega}_{q^2} \qbinom{r+s+1}{r}_{q^2}\notag\\
	&\quad+ \sum_{\substack{r,s=-\infty\\ r\not\equiv s \bmod 2}}^\infty \frac{q^{r^2+r+s^2+s+\omega}}{(q^2;q^2)_{r+s}}\qbinom{r+s}{r+\omega}_{q^2} \qbinom{r+s+1}{r}_{q^2}.
\end{align}

First, we deduce from \eqref{eq:G1-omega-rs} that
\begin{align}
	&G_1(x)+G_1(x^{-1})\notag\\
	&\qquad=\sum_{\omega=-\infty}^\infty \big(x^{\omega}+x^{-\omega}\big)\sum_{r,s=-\infty}^\infty \frac{q^{r^2+r+s^2+s+\omega}}{(q^2;q^2)_{r+s}}\qbinom{r+s}{r+\omega}_{q^2} \qbinom{r+s+1}{r}_{q^2}.
\end{align}
Throughout, we introduce a family of auxiliary series
\begin{align}
	H^+_\omega=H^+_\omega(q):=\sum_{r,s=-\infty}^\infty \frac{q^{r^2+r+s^2+s+\omega}}{(q^2;q^2)_{r+s}}\qbinom{r+s}{r+\omega}_{q^2} \qbinom{r+s+1}{r}_{q^2}.
\end{align}
Then,
\begin{align*}
	H^+_\omega &= \sum_{s=-\infty}^\infty \frac{q^{s^2+s+\omega}}{(q^2;q^2)_{s-\omega} (q^2;q^2)_{s+1}}\sum_{r\ge 0} \frac{q^{r^2+r}(q^2;q^2)_{r+s+1}}{(q^2;q^2)_{r+\omega} (q^2;q^2)_r}\\
	&= \lim_{u\to 1}\sum_{s=-\infty}^\infty \frac{q^{s^2+s+\omega}}{(q^2;q^2)_{s-\omega} (uq^2;q^2)_{\omega}} \sum_{r\ge 0}\frac{q^{2\binom{r}{2}+2r}(q^{4+2s};q^2)_r}{(uq^{2+2\omega};q^2)_r (q^2;q^2)_r}\\
	&= \lim_{u\to 1}\sum_{s=-\infty}^\infty \frac{q^{s^2+s+\omega}}{(q^2;q^2)_{s-\omega} (uq^2;q^2)_{\omega}} \lim_{\tau\to 0} {}_{2}\phi_{1}\left(\begin{matrix} q^{4+2s},1/\tau\\ uq^{2+2\omega} \end{matrix}; q^2, -q^2 \tau\right)\\
	\text{\tiny (by \eqref{eq:Heine-2})}&=\frac{(-q^2;q^2)_\infty}{(q^2;q^2)_\infty}\sum_{s=-\infty}^\infty \frac{q^{s^2+s+\omega}}{(q^2;q^2)_{s-\omega}} \lim_{\tau\to 0} {}_{2}\phi_{1}\left(\begin{matrix} -q^{4+2s-2\omega},1/\tau\\ -q^{2} \end{matrix}; q^2, q^{2+2\omega} \tau\right)\\
	&=\frac{(-q^2;q^2)_\infty}{(q^2;q^2)_\infty}\sum_{r,s=-\infty}^\infty \frac{q^{s^2+s+\omega}}{(q^2;q^2)_{s-\omega}} \frac{(-1)^r q^{r^2+r+2r\omega}(-q^{4+2s-2\omega};q^2)_r}{(q^4;q^4)_r}.
\end{align*}
Notice that we add the limit $u\to 1$ to make our computation more rigorous by avoiding the zero denominator in $(q^{2+2\omega};q^2)_r$ when $\omega<0$. However, since this step is more or less straightforward, it will be omitted in our subsequent calculations. Now, letting $s\mapsto s+\omega-1$ in the above yields
\begin{align*}
	H^+_{\omega} &= \frac{q^{\omega^2}(-q^2;q^2)_\infty}{(q^2;q^2)_\infty} \sum_{r,s\ge 0} \frac{(-1)^r q^{r^2+r+s^2-s+2r\omega+2s\omega}(1-q^{2s})(-q^{2};q^2)_{r+s}}{(q^4;q^4)_r(q^4;q^4)_s}\\
	&= \frac{q^{\omega^2}(-q^2;q^2)_\infty}{(q^2;q^2)_\infty} \sum_{M\ge 1} \frac{q^{2M\omega}}{(q^2;q^2)_M}\sum_{\substack{r,s\ge 0\\ r+s=M}} (-1)^r q^{r^2+r+s^2-s}(1-q^{2s}) \qbinom{M}{r}_{q^4}.
\end{align*}
Notice that for $M\ge 1$,
\begin{align*}
	&\sum_{\substack{r,s\ge 0\\ r+s=M}} (-1)^r q^{r^2+r+s^2-s}(1-q^{2s}) \qbinom{M}{r}_{q^4}\\
	&=q^{M^2-M}\sum_{r\ge 0}(-1)^r q^{4\binom{r}{2}+(4-2M)r} \qbinom{M}{r}_{q^4}\\
	&\quad-q^{M^2+M}\sum_{r\ge 0}(-1)^r q^{4\binom{r}{2}+(2-2M)r} \qbinom{M}{r}_{q^4}\\
	\text{\tiny (by \eqref{eq:q-Bin})}&=q^{M^2-M} (q^{4-2M};q^4)_M-q^{M^2+M}(q^{2-2M};q^4)_M\\
	&=\begin{cases}
		-(-1)^N q^{2N^2+2N} (q^2;q^4)_N^2, & \text{if $M=2N$},\\
		(-1)^N q^{2N^2+2N} (q^2;q^4)_N(q^2;q^4)_{N+1}, & \text{if $M=2N+1$}.
	\end{cases}
\end{align*}
Thus,
\begin{align*}
	H^+_\omega= \frac{q^{\omega^2}(-q^2;q^2)_\infty}{(q^2;q^2)_\infty}\left(1-(1-q^{2\omega})\sum_{N\ge 0}\frac{(-1)^N q^{2N^2+2N+4N\omega} (q^2;q^4)_N}{(q^4;q^4)_N}\right).
\end{align*}
Now, we rewrite the summation over $N$ as follows,
\begin{align*}
	\sum_{N\ge 0}\frac{(-1)^N q^{2N^2+2N+4N\omega} (q^2;q^4)_N}{(q^4;q^4)_N}&=\lim_{\tau\to 0} {}_{2}\phi_{1}\left(\begin{matrix} q^2,1/\tau\\ \tau \end{matrix}; q^4, q^{4+4\omega} \tau\right)\\
	\text{\tiny (by \eqref{eq:Heine-2})}&= (q^{4+4\omega};q^4)_\infty \lim_{\tau\to 0} {}_{2}\phi_{1}\left(\begin{matrix} q^{6+4\omega}/\tau,1/\tau\\ q^{4+4\omega} \end{matrix}; q^4, \tau^2\right)\\
	&=(q^{4+4\omega};q^4)_\infty \sum_{N=-\infty}^\infty \frac{q^{4N^2+2N+4N\omega}}{(q^4;q^4)_N (q^{4+4\omega};q^4)_N}\\
	&=(q^4;q^4)_\infty \sum_{N=-\infty}^\infty \frac{q^{4N^2+2N+4N\omega}}{(q^4;q^4)_N (q^4;q^4)_{N+\omega}}.
\end{align*}
Thus,
\begin{align*}
	H^+_\omega = \frac{q^{\omega^2}(-q^2;q^2)_\infty}{(q^2;q^2)_\infty}\left(1-(q^4;q^4)_\infty (1-q^{2\omega}) \sum_{N=-\infty}^\infty \frac{q^{4N^2+2N+4N\omega}}{(q^4;q^4)_N (q^4;q^4)_{N+\omega}}\right).
\end{align*}
Also,
\begin{align*}
	H^+_{-\omega}&=\frac{q^{\omega^2}(-q^2;q^2)_\infty}{(q^2;q^2)_\infty}\left(1-(q^4;q^4)_\infty (1-q^{-2\omega}) \sum_{N=-\infty}^\infty \frac{q^{4N^2+2N-4N\omega}}{(q^4;q^4)_N (q^4;q^4)_{N-\omega}}\right)\\
	&=\frac{q^{\omega^2}(-q^2;q^2)_\infty}{(q^2;q^2)_\infty}\left(1-(q^4;q^4)_\infty (1-q^{-2\omega}) \sum_{N=-\infty}^\infty \frac{q^{4(N+\omega)^2+2(N+\omega)-4(N+\omega)\omega}}{(q^4;q^4)_{N+\omega} (q^4;q^4)_{N}}\right)\\
	&=\frac{q^{\omega^2}(-q^2;q^2)_\infty}{(q^2;q^2)_\infty}\left(1-(q^4;q^4)_\infty (q^{2\omega}-1) \sum_{N=-\infty}^\infty \frac{q^{4N^2+2N+4N\omega}}{(q^4;q^4)_N (q^4;q^4)_{N+\omega}}\right).
\end{align*}
We conclude that for any integer $\omega$,
\begin{align}\label{eq:H-relation}
	H^+_\omega+H^+_{-\omega} = \frac{2q^{\omega^2}(-q^2;q^2)_\infty}{(q^2;q^2)_\infty}.
\end{align}
Therefore,
\begin{align}\label{eq:G1-sum}
	G_1(x)+G_1(x^{-1}) &= \sum_{\omega=-\infty}^\infty \big(x^{\omega}+x^{-\omega}\big) H^+_{\omega}\notag\\
	&=\sum_{\omega=-\infty}^\infty x^{\omega} \big(H^+_{\omega}+H^+_{-\omega}\big) \notag\\
	\text{\tiny (by \eqref{eq:H-relation})}&=\sum_{\omega=-\infty}^\infty x^{\omega} \cdot  \frac{2q^{\omega^2}(-q^2;q^2)_\infty}{(q^2;q^2)_\infty}\notag\\
	\text{\tiny (by \eqref{eq:JTP})}&=\frac{2(-q^2;q^2)_{\infty}(-xq,-x^{-1}q,q^2;q^2)_\infty}{(q^2;q^2)_{\infty}}.
\end{align}

On the other hand, it also follows from \eqref{eq:G1-omega-rs} that
\begin{align}
	&G_1(x)-G_1(x^{-1})\notag\\
	&\qquad=\sum_{\omega=-\infty}^\infty \big(x^{\omega}-x^{-\omega}\big)\sum_{r,s=-\infty}^\infty \frac{(-1)^{r+s}q^{r^2+r+s^2+s+\omega}}{(q^2;q^2)_{r+s}}\qbinom{r+s}{r+\omega}_{q^2} \qbinom{r+s+1}{r}_{q^2}.
\end{align}
Now, we define
\begin{align}
	H^-_\omega=H^-_\omega(q):=\sum_{r,s=-\infty}^\infty \frac{(-1)^{r+s}q^{r^2+r+s^2+s+\omega}}{(q^2;q^2)_{r+s}}\qbinom{r+s}{r+\omega}_{q^2} \qbinom{r+s+1}{r}_{q^2}.
\end{align}
Then,
\begin{align*}
	H^-_\omega &= \sum_{s=-\infty}^\infty \frac{(-1)^s q^{s^2+s+\omega}}{(q^2;q^2)_{s-\omega} (q^2;q^2)_{s+1}}\sum_{r\ge 0} \frac{(-1)^r q^{r^2+r}(q^2;q^2)_{r+s+1}}{(q^2;q^2)_{r+\omega} (q^2;q^2)_r}\\
	&= \sum_{s=-\infty}^\infty \frac{(-1)^s q^{s^2+s+\omega}}{(q^2;q^2)_{s-\omega} (q^2;q^2)_{\omega}} \sum_{r\ge 0}\frac{(-1)^r q^{2\binom{r}{2}+2r}(q^{4+2s};q^2)_r}{(q^{2+2\omega};q^2)_r (q^2;q^2)_r}\\
	&= \sum_{s=-\infty}^\infty \frac{(-1)^s q^{s^2+s+\omega}}{(q^2;q^2)_{s-\omega} (q^2;q^2)_{\omega}} \lim_{\tau\to 0} {}_{2}\phi_{1}\left(\begin{matrix} q^{4+2s},1/\tau\\ q^{2+2\omega} \end{matrix}; q^2, q^2 \tau\right)\\
	\text{\tiny (by \eqref{eq:Heine-2})}&=\sum_{r,s=-\infty}^\infty \frac{(-1)^{r+s} q^{s^2+s+\omega}}{(q^2;q^2)_{s-\omega}} \frac{q^{r^2+r+2r\omega}(q^{4+2s-2\omega};q^2)_r}{(q^2;q^2)_r^2}\\
	\text{\tiny ($s\mapsto s+\omega-1$)}&= -(-1)^{\omega}q^{\omega^2} \sum_{r,s\ge 0} \frac{(-1)^{r+s} q^{r^2+r+s^2-s+2r\omega+2s\omega}(1-q^{2s})(q^{2};q^2)_{r+s}}{(q^2;q^2)_r^2(q^2;q^2)_s^2}.
\end{align*}
We also have
\begin{align*}
	H^-_{-\omega}&=\sum_{r=-\infty}^\infty \frac{(-1)^r q^{r^2+r-\omega}}{(q^2;q^2)_{r-\omega} (q^2;q^2)_r} \sum_{s\ge -1} \frac{(-1)^s q^{s^2+s} (q^2;q^2)_{r+s+1}}{(q^2;q^2)_{s+\omega} (q^2;q^2)_{s+1}}\\
	&=-\sum_{r=-\infty}^\infty \frac{(-1)^r q^{r^2+r-\omega}}{(q^2;q^2)_{r-\omega} (q^2;q^2)_r} \sum_{s\ge 0} \frac{(-1)^s q^{s^2-s} (q^2;q^2)_{r+s}}{(q^2;q^2)_{s+\omega-1} (q^2;q^2)_{s}}\\
	&=-\sum_{r=-\infty}^\infty \frac{(-1)^r q^{r^2+r-\omega}}{(q^2;q^2)_{r-\omega} (q^2;q^2)_{\omega-1}}  \lim_{\tau\to 0} {}_{2}\phi_{1}\left(\begin{matrix} q^{2+2r},1/\tau\\ q^{2\omega} \end{matrix}; q^2, \tau\right)\\
	\text{\tiny (by \eqref{eq:Heine-2})}&= -\sum_{r=-\infty}^\infty \frac{(-1)^r q^{r^2+r-\omega}}{(q^2;q^2)_{r-\omega}}\sum_{s\ge 1}\frac{(-1)^s q^{s^2-s+2s\omega} (q^{2+2r-2\omega};q^2)_s}{(q^2;q^2)_s(q^2;q^2)_{s-1}}\\
	&=-\sum_{r=-\infty}^\infty \frac{(-1)^r q^{r^2+r-\omega}}{(q^2;q^2)_{r-\omega}}\sum_{s\ge 0}\frac{(-1)^s q^{s^2-s+2s\omega} (1-q^{2s})(q^{2+2r-2\omega};q^2)_s}{(q^2;q^2)_s(q^2;q^2)_{s}}\\
	\text{\tiny ($r\mapsto r+\omega$)}&= -(-1)^{\omega}q^{\omega^2} \sum_{r,s\ge 0} \frac{(-1)^{r+s} q^{r^2+r+s^2-s+2r\omega+2s\omega}(1-q^{2s})(q^{2};q^2)_{r+s}}{(q^2;q^2)_r^2(q^2;q^2)_s^2}.
\end{align*}
Thus, for any integer $\omega$,
\begin{align}\label{eq:H^*-relation}
	H^-_\omega=H^-_{-\omega}.
\end{align}
We conclude that
\begin{align}\label{eq:G1-diff}
	G_1(x)-G_1(x^{-1})&= \sum_{\omega=-\infty}^\infty \big(x^{\omega}-x^{-\omega}\big) H^-_{\omega}\notag\\
	&=\sum_{\omega=-\infty}^\infty x^{\omega} \big(H^-_{\omega}-H^-_{-\omega}\big) \notag\\
	\text{\tiny (by \eqref{eq:H^*-relation})}&=0.
\end{align}

Finally, summing \eqref{eq:G1-sum} and \eqref{eq:G1-diff} yields \eqref{eq:Lambda12-bi-gf}. 

\subsection{Proof of \eqref{eq:Lambda22-bi-gf}}

Recall that
\begin{align*}
	G_2(x):=G_2(x,q)=\sum_{\omega=-\infty}^\infty \sum_{n\ge 0} |\Lambda_{\omega}^{2,2}(n)| x^{\omega}q^{n}.
\end{align*}
In \eqref{g2omega}, letting $r=n+k$ and $s=n-k$ in the first double summation and letting $r=n-k+1$ and $s=n+k$ in the second double summation, we have
\begin{align}\label{eq:G2-omega-rs}
	[x^{\omega}]G_2(x)
	&= \sum_{\substack{r,s=-\infty\\ r\equiv s \bmod 2}}^\infty \frac{q^{r^2+s^2+2s+\omega}}{(q^2;q^2)_{r+s}}\qbinom{r+s}{r-\omega}_{q^2} \qbinom{r+s}{r}_{q^2}\notag\\
	&\quad+ \sum_{\substack{r,s=-\infty\\ r\not\equiv s \bmod 2}}^\infty \frac{q^{r^2+s^2+2s+1-\omega}}{(q^2;q^2)_{r+s}}\qbinom{r+s}{r-1+\omega}_{q^2} \qbinom{r+s}{r}_{q^2}.
\end{align}

We start with the difference of $G_2(x)$ and $xG_2(x^{-1})$. It follows from \eqref{eq:G2-omega-rs} that
\begin{align}
	&G_2(x)-xG_2(x^{-1})\notag\\
	&\qquad=\sum_{\omega=-\infty}^\infty \big(x^{\omega}-x^{1-\omega}\big)\sum_{r,s=-\infty}^\infty \frac{(-1)^{r+s}q^{r^2+s^2+2s+\omega}}{(q^2;q^2)_{r+s}}\qbinom{r+s}{r-\omega}_{q^2} \qbinom{r+s}{r}_{q^2}.
\end{align}
We define
\begin{align}
	I^-_\omega=I^-_\omega(q):=\sum_{r,s=-\infty}^\infty \frac{(-1)^{r+s}q^{r^2+s^2+2s+\omega}}{(q^2;q^2)_{r+s}}\qbinom{r+s}{r-\omega}_{q^2} \qbinom{r+s}{r}_{q^2}.
\end{align}
Then,
\begin{align*}
	I^-_\omega &= \sum_{s=-\infty}^\infty \frac{(-1)^s q^{s^2+2s+\omega}}{(q^2;q^2)_{s+\omega} (q^2;q^2)_{s}}\sum_{r\ge 0} \frac{(-1)^r q^{r^2}(q^2;q^2)_{r+s}}{(q^2;q^2)_{r-\omega} (q^2;q^2)_r}\\
	&= \sum_{s=-\infty}^\infty \frac{(-1)^s q^{s^2+2s+\omega}}{(q^2;q^2)_{s+\omega} (q^2;q^2)_{-\omega}} \sum_{r\ge 0}\frac{(-1)^r q^{2\binom{r}{2}+r}(q^{2+2s};q^2)_r}{(q^{2-2\omega};q^2)_r (q^2;q^2)_r}\\
	&= \sum_{s=-\infty}^\infty \frac{(-1)^s q^{s^2+2s+\omega}}{(q^2;q^2)_{s+\omega} (q^2;q^2)_{-\omega}} \lim_{\tau\to 0} {}_{2}\phi_{1}\left(\begin{matrix} q^{2+2s},1/\tau\\ q^{2-2\omega} \end{matrix}; q^2, q \tau\right)\\
	\text{\tiny (by \eqref{eq:Heine-2})}&=\frac{(q;q^2)_\infty}{(q^2;q^2)_\infty}\sum_{s=-\infty}^\infty \frac{(-1)^s q^{s^2+2s+\omega}}{(q^2;q^2)_{s+\omega}} \lim_{\tau\to 0} {}_{2}\phi_{1}\left(\begin{matrix} q^{1+2s+2\omega},1/\tau\\ q \end{matrix}; q^2, q^{2-2\omega} \tau\right)\\
	&=\frac{(q;q^2)_\infty}{(q^2;q^2)_\infty}\sum_{r,s=-\infty}^\infty \frac{(-1)^s q^{s^2+2s+\omega}}{(q^2;q^2)_{s+\omega}} \frac{(-1)^r q^{r^2+r-2r\omega}(q^{1+2s+2\omega};q^2)_r}{(q^2;q^2)_r(q;q^2)_r}.
\end{align*}
Letting $s\mapsto s-\omega$ in the above yields
\begin{align*}
	I^-_{\omega} &= \frac{(-1)^{\omega}q^{\omega^2-\omega}(q;q^2)_\infty}{(q^2;q^2)_\infty} \sum_{r,s\ge 0} \frac{(-1)^{r+s} q^{r^2+r+s^2+2s-2r\omega-2s\omega}(q;q^2)_{r+s}}{(q^2;q^2)_r(q;q^2)_r(q^2;q^2)_s(q;q^2)_s}\\
	&=\frac{(-1)^{\omega}q^{\omega^2-\omega}(q;q^2)_\infty}{(q^2;q^2)_\infty} + \frac{(-1)^{\omega}q^{\omega^2-\omega}(q;q^2)_\infty}{(q^2;q^2)_\infty}\sum_{M\ge 1} (-1)^M q^{-2M\omega} (q;q^2)_M S_{1,M},
\end{align*}
where
\begin{align*}
	S_{1,M}:=\sum_{\substack{r,s\ge 0\\ r+s=M}} \frac{q^{r^2+r+s^2+2s}}{(q^2;q^2)_r(q;q^2)_r(q^2;q^2)_s(q;q^2)_s}.
\end{align*}
Similarly,
\begin{align*}
	I^-_{1-\omega}&= \sum_{r=-\infty}^\infty \frac{(-1)^r q^{r^2-\omega+1}}{(q^2;q^2)_{r+\omega-1} (q^2;q^2)_r}\sum_{s\ge 0} \frac{(-1)^s q^{s^2+2s}(q^2;q^2)_{r+s}}{(q^2;q^2)_{s-\omega+1} (q^2;q^2)_{s}}\\
	&= \sum_{r=-\infty}^\infty \frac{(-1)^r q^{r^2-\omega+1}}{(q^2;q^2)_{r+\omega-1} (q^2;q^2)_{1-\omega}} \sum_{s\ge 0} \frac{(-1)^s q^{2\binom{s}{2}+3s}(q^{2+2r};q^2)_s}{(q^{4-2\omega};q^2)_{s} (q^2;q^2)_{s}}\\
	&= \sum_{r=-\infty}^\infty \frac{(-1)^r q^{r^2-\omega+1}}{(q^2;q^2)_{r+\omega-1} (q^2;q^2)_{1-\omega}} \lim_{\tau\to 0} {}_{2}\phi_{1}\left(\begin{matrix} q^{2+2r},1/\tau\\ q^{4-2\omega} \end{matrix}; q^2, q^3 \tau\right)\\
	\text{\tiny (by \eqref{eq:Heine-2})}&=\frac{(q^3;q^2)_\infty}{(q^2;q^2)_\infty}\sum_{r=-\infty}^\infty \frac{(-1)^r q^{r^2-\omega+1}}{(q^2;q^2)_{r+\omega-1}} \lim_{\tau\to 0} {}_{2}\phi_{1}\left(\begin{matrix} q^{1+2r+2\omega},1/\tau\\ q^3 \end{matrix}; q^2, q^{4-2\omega} \tau\right)\\
	&=\frac{(q^3;q^2)_\infty}{(q^2;q^2)_\infty}\sum_{r,s=-\infty}^\infty \frac{(-1)^r q^{r^2-\omega+1}}{(q^2;q^2)_{r+\omega-1}} \frac{(-1)^s q^{s^2+3s-2s\omega}(q^{1+2r+2\omega};q^2)_s}{(q^2;q^2)_s(q^3;q^2)_s}\\
	&=\frac{(q;q^2)_\infty}{(q^2;q^2)_\infty}\sum_{r,s=-\infty}^\infty \frac{(-1)^r q^{r^2-\omega+1}}{(q^2;q^2)_{r+\omega-1}} \frac{(-1)^s q^{s^2+3s-2s\omega}(q^{1+2r+2\omega};q^2)_s}{(q^2;q^2)_s(q;q^2)_{s+1}}.
\end{align*}
Letting $r\mapsto r-\omega+1$ and $s\mapsto s-1$ in the above yields
\begin{align*}
	I^-_{1-\omega} &= \frac{(-1)^{\omega}q^{\omega^2-\omega}(q;q^2)_\infty}{(q^2;q^2)_\infty}\sum_{r,s\ge 0} \frac{(-1)^{r+s} q^{r^2+2r+s^2+s-2r\omega-2s\omega}(q;q^2)_{r+s}}{(q^2;q^2)_r(q;q^2)_{r+1}(q^2;q^2)_{s-1}(q;q^2)_s}\\
	&=\frac{(-1)^{\omega}q^{\omega^2-\omega}(q;q^2)_\infty}{(q^2;q^2)_\infty}\sum_{M\ge 1} (-1)^M q^{-2M\omega} (q;q^2)_M S_{2,M},
\end{align*}
where
\begin{align*}
S_{2,M}:=\sum_{\substack{r,s\ge 0\\ r+s=M}} \frac{q^{r^2+2r+s^2+s}}{(q^2;q^2)_r(q;q^2)_{r+1}(q^2;q^2)_{s-1}(q;q^2)_s}.
\end{align*}
We claim that for $M\ge 1$,
\begin{align}
	S_{1,M}=S_{2,M}=\frac{q^{M(M+3)/2}(-q;q)_{M-1}}{(q;q^2)_{M}(q;q)_M}.
\end{align}
To see this, we have
\begin{align*}
	S_{1,M}&=\sum_{\substack{r,s\ge 0\\ r+s=M}} \frac{q^{r^2+r+s^2+2s}}{(q^2;q^2)_r(q;q^2)_r(q^2;q^2)_s(q;q^2)_s}\\
	&=\frac{q^{M^2+2M}}{(q;q^2)_M(q^2;q^2)_M}{}_{2}\phi_{1}\left(\begin{matrix} q^{-2M},q^{1-2M}\\ q \end{matrix}; q^2, q^{2M}\right)\\
	\text{\tiny (by \eqref{71+})}&=\frac{q^{M^2+2M}}{(q;q^2)_M(q^2;q^2)_M}\frac{(-q^{-M},q^M;q)_\infty}{2(q^{2M};q^2)_\infty}\\
	&=\frac{q^{M(M+3)/2}(-q;q)_{M-1}}{(q;q^2)_{M}(q;q)_M},
\end{align*}
and
\begin{align*}
	S_{2,M}&=\sum_{\substack{r,s\ge 0\\ r+s=M}} \frac{q^{r^2+2r+s^2+s}}{(q^2;q^2)_r(q;q^2)_{r+1}(q^2;q^2)_{s-1}(q;q^2)_s}\\
	&=\lim_{u\to 1}\frac{1-u}{1-q}\sum_{\substack{r,s\ge 0\\ r+s=M}} \frac{q^{r^2+2r+s^2+s}}{(q^2;q^2)_r(q^3;q^2)_{r}(u;q^2)_{s}(q;q^2)_s}\\
	&=\frac{q^{M^2+M}}{(1-q)(q;q^2)_M(q^2;q^2)_{M-1}}{}_{2}\phi_{1}\left(\begin{matrix} q^{2-2M},q^{1-2M}\\ q^3 \end{matrix}; q^2, q^{2M}\right)\\
	\text{\tiny (by \eqref{71+})}&=\frac{q^{M^2+M}}{(q;q^2)_M(q^2;q^2)_{M-1}}\frac{(-q^{1-M},q^{1+M};q)_\infty}{2(q^{2M};q^2)_\infty}\\
	&=\frac{q^{M(M+3)/2}(-q;q)_{M-1}}{(q;q^2)_{M}(q;q)_M},
\end{align*}
thereby confirming the claim. It follows that for any integer $\omega$,
\begin{align}\label{eq:I^*-relation}
	I^-_{\omega}-I^-_{1-\omega}=\frac{(-1)^{\omega}q^{\omega^2-\omega}(q;q^2)_\infty}{(q^2;q^2)_\infty}.
\end{align}
We conclude that
\begin{align}\label{eq:G2-diff}
	G_2(x)-xG_2(x^{-1}) &= \sum_{\omega=-\infty}^\infty \big(x^{\omega}-x^{1-\omega}\big) I^-_{\omega}\notag\\
	&=\sum_{\omega=-\infty}^\infty x^{\omega} \big(I^-_{\omega}-I^-_{1-\omega}\big) \notag\\
	\text{\tiny (by \eqref{eq:I^*-relation})}&=\sum_{\omega=-\infty}^\infty x^{\omega} \cdot  \frac{(-1)^{\omega}q^{\omega^2-\omega}(q;q^2)_\infty}{(q^2;q^2)_\infty}\notag\\
	\text{\tiny (by \eqref{eq:JTP})} &= \frac{(q;q^2)_\infty(x,x^{-1}q^2,q^2;q^2)_\infty}{(q^2;q^2)_\infty}.
\end{align}

On the other hand, we deduce from \eqref{eq:G2-omega-rs} that
\begin{align}
	&G_2(x)+xG_2(x^{-1})\notag\\
	&\qquad=\sum_{\omega=-\infty}^\infty \big(x^{\omega}+x^{1-\omega}\big)\sum_{r,s=-\infty}^\infty \frac{q^{r^2+s^2+2s+\omega}}{(q^2;q^2)_{r+s}}\qbinom{r+s}{r-\omega}_{q^2} \qbinom{r+s}{r}_{q^2}.
\end{align}
We define
\begin{align}
	I^+_\omega=I^+_\omega(q):=\sum_{r,s=-\infty}^\infty \frac{q^{r^2+s^2+2s+\omega}}{(q^2;q^2)_{r+s}}\qbinom{r+s}{r-\omega}_{q^2} \qbinom{r+s}{r}_{q^2}.
\end{align}
Note that
\begin{align*}
    I^+_\omega(q) = (-1)^\omega I^-_\omega(-q).
\end{align*}
It follows that for any integer $\omega$,
\begin{align}\label{eq:I-relation}
	I^+_{\omega}+I^+_{1-\omega}&=(-1)^\omega I^-_\omega(-q) + (-1)^{1-\omega} I^-_{1-\omega}(-q)\notag\\
	&=(-1)^{\omega}\big(I^-_\omega(-q)-I^-_{1-\omega}(-q)\big)\notag\\
	\text{\tiny (by \eqref{eq:I^*-relation})}&=\frac{q^{\omega^2-\omega}(-q;q^2)_\infty}{(q^2;q^2)_\infty}.
\end{align}
We conclude that
\begin{align}\label{eq:G2-sum}
	G_2(x)+xG_2(x^{-1}) &= \sum_{\omega=-\infty}^\infty \big(x^{\omega}+x^{1-\omega}\big) I^+_{\omega}\notag\\
	&=\sum_{\omega=-\infty}^\infty x^{\omega} \big(I^+_{\omega}+I^+_{1-\omega}\big) \notag\\
	\text{\tiny (by \eqref{eq:I-relation})}&=\sum_{\omega=-\infty}^\infty x^{\omega} \cdot  \frac{q^{\omega^2-\omega}(-q;q^2)_\infty}{(q^2;q^2)_\infty}\notag\\
	\text{\tiny (by \eqref{eq:JTP})} &= \frac{(-q;q^2)_\infty(-x,-x^{-1}q^2,q^2;q^2)_\infty}{(q^2;q^2)_\infty}.
\end{align}

Finally, summing \eqref{eq:G2-diff} and \eqref{eq:G2-sum} yields \eqref{eq:Lambda22-bi-gf}. 

\subsection{Proof of Corollary~\ref{coro:1.4}}

For notational convenience, we rewrite the identities in the corollary as follows: 
\begin{equation}\label{coro-number1}
\sum_{n\geq 0}|\Lambda_\omega^{1,2}(n)|q^n=\frac{q^{\omega^2}(-q^2;q^2)_\infty}{(q^2;q^2)_\infty}
\end{equation}
and
\begin{equation}\label{coro-number2}
\sum_{n\geq 0}|\Lambda_\omega^{2,2}(n)|q^n=\frac{q^{\omega^2-\omega}(-q;q^2)_\infty}{(q^2;q^2)_\infty}.
\end{equation}
If we expand the right hand side of \eqref{eq:Lambda12-bi-gf} using Jacobi's triple product identity \eqref{eq:JTP}, then it becomes
\begin{equation*}
\frac{(-q^2;q^2)_{\infty}}{(q^2;q^2)_{\infty}} \sum_{\omega =-\infty}^{\infty} x^{\omega} q^{\omega^2}.
\end{equation*}
Equating the coefficient of $x^{\omega}$, we deduce \eqref{coro-number1}. 

Similarly, \eqref{coro-number2} follows from \eqref{eq:Lambda12-bi-gf} with an application of Jacobi's triple product identity \eqref{eq:JTP} to its right side and after some cancellations.  We omit the details.


\section{The generating function for general $\Lambda^{a,m}$}\label{sec5}

 In this section, we discuss the generating function for $\Lambda^{a,m}$. Our main goal is to give a $q$-series proof of Theorem~\ref{thm:AM}. 

We start by recalling the conditions for $\pi=(\pi^{(1)},\ldots, \pi^{(m)})\in \Lambda^{a,m }$ from Section~\ref{sec:2-restricted}:
\begin{enumerate}
	\item[(i)] each $\pi^{(i)}$ is a strict partition;
	\item[(ii)] $\pi^{(i)}_1\le \ell(\pi^{(i+1)})$ for $i\neq a$ and $\pi^{(a)}_1\le \ell(\pi^{(a+1)})+1$.  Here, if $a=m$, then $\ell(\pi^{(a+1)})=\infty$. 
\end{enumerate}

Our first observation is that the generating function for $m$-multipartitions in $\Lambda^{a,m}$ can be represented as a multisum. Moreover, although $\Lambda^{a,m}$ is defined for $m\ge 2$, the following result still holds for $m=1$, so we will assume that $m\ge 1$ throughout this section. In particular, $\Lambda^{1,1}$ is just the set of strict partitions, i.e., partitions into distinct parts. Another note to make is that we can even define $\Lambda^{0,m}$ by taking the sequence $$t_1=\cdots =t_m=1,$$ but this set is exactly $\Lambda^{m,m}$.

\begin{theorem}\label{th:gen-sum}
	For $m\ge 1$ and $0 \le a\le m$,
	\begin{align}
		&\sum_{n\ge 0} |\Lambda^{a,m}(n)| q^n \notag \\
		&=\sum_{N_1,\ldots, N_m\ge 0}\frac{q^{\sum_{i=1}^{m}\binom{N_i+1}{2}}}{(q;q)_{N_m}}\qbinom{N_2+\delta_{a+1,2}}{N_1}\qbinom{N_3+\delta_{a+1,3}}{N_2}\cdots \qbinom{N_{m}+\delta_{a+1,m}}{N_{m-1}}\label{gen-q-binomial}\\
		&=\sum_{\substack{N_m\ge N_{m-1}\ge \cdots \ge N_{a+1} \\ N_{a+1}+1\ge N_a\ge \cdots \ge N_2\ge 0}} \frac{q^{\sum_{i=2}^{m} \binom{N_i+1}{2}} (-q;q)_{N_{2+\delta_{a+1,2}}} \big(1-q^{N_{a+1}+1}\big) }{ \big(1-q^{N_{a+1}-N_{a}+1} \big) \prod_{i=2}^{m}(q;q)_{N_i-N_{i-1}}}. \label{gen1}
	\end{align}
	Here, $\delta$ is the Kronecker delta function. Also, in \eqref{gen1}, we take $N_1=0$ and $N_{m+1}=\infty$.
\end{theorem}

\begin{proof}
	We make use of the fact that the generating function for strict partitions with exactly $M$ parts and largest part at most $N$ is
	\begin{align*}
		q^{\binom{M+1}{2}}\qbinom{N}{M}.
	\end{align*}
	Now, assume that each $\pi^{(i)}$ has exactly $N_i$ parts. Then, for each $1\le i\le m-1$, the strict partition $\pi^{(i)}$ has $N_i$ parts with its largest part at most $N_{i+1}+\delta_{a+1,i+1}$, and therefore the generating function is 
	\begin{align*}
		q^{\binom{N_{i}+1}{2}}\qbinom{N_{i+1}+\delta_{a+1,i+1}}{N_i}.
	\end{align*}
	Furthermore, since the strict partition $\pi^{(m)}$ has exactly $N_m$ parts, its generating function is
	\begin{align*}
		q^{\binom{N_{m}+1}{2}}\frac{1}{(q;q)_{N_m}}.
	\end{align*}
	Thus, \eqref{gen-q-binomial} follows. For \eqref{gen1}, we calculate the inner summation in \eqref{gen-q-binomial} over $N_1$ by \eqref{eq:q-Bin}. Then from \eqref{gen-q-binomial},
	\begin{align*}
		&\sum_{n\ge 0} |\Lambda^{a,m}(n)| q^n \\
		&=\sum_{N_m,\ldots, N_2\ge 0}\frac{q^{\sum_{i=2}^{m}\binom{N_i+1}{2}}(-q;q)_{N_{2}+\delta_{a+1,2} }}{(q;q)_{N_m}}\qbinom{N_3+\delta_{a+1,3}}{N_2}\qbinom{N_4+\delta_{a+1,4}}{N_3}\cdots \qbinom{N_{m}+\delta_{a+1,m}}{N_{m-1}},
	\end{align*}
and \eqref{gen1} follows by rewriting the $q$-binomial coefficients using the $q$-Pochhammer symbols.
\end{proof}

For notational convenience, we reverse the order of $\pi^{(i)}$, i.e., replacing $i$ by $m+1-i$ in \eqref{gen-q-binomial}, which also affects the order of $t_i$'s in ${\bf t}=(t_1,\ldots, t_m)$. Namely, we have to replace $a$ by $m+1-a$.
Then the generating function in Theorem \ref{th:gen-sum} becomes
\begin{align}
	&\sum_{n\ge 0} |\Lambda^{a, m}(n) | q^n \notag\\
	&=\sum_{N_1,\ldots, N_m\ge 0}\frac{q^{\sum_{i=1}^{m}\binom{N_i+1}{2}}}{(q;q)_{N_1}}\qbinom{N_1+\delta_{m-a,1}}{N_2}\qbinom{N_2+\delta_{m-a,2}}{N_3}\cdots \qbinom{N_{m-1}+\delta_{m-a,m-1}}{N_{m}}. \notag 
\end{align}

Thus, Theorem~\ref{thm:AM} is equivalent to the following theorem.
\begin{theorem}\label{thm2.1}
	For $m\ge 1$ and $0 \le a\le m$,
		\begin{align}\label{eq:gen4}
		&\frac{(q^{a+1}, q^{m+1-a}, q^{m+2};q^{m+2})_{\infty}}{(q;q)_{\infty} (q;q^2)_{\infty}}\notag\\
		&\qquad=\sum_{N_1,\ldots, N_m\ge 0}\frac{q^{\sum_{i=1}^{m}\binom{N_i+1}{2}}}{(q;q)_{N_1}}\qbinom{N_1+\delta_{m-a,1}}{N_2}\qbinom{N_2+\delta_{m-a,2}}{N_3}\cdots \qbinom{N_{m-1}+\delta_{m-a,m-1}}{N_{m}}.
	\end{align}
\end{theorem}

Interestingly, we have a piece of evidence from an identity due to Andrews \cite[Theorem 3, eq.~(1.9) with $a=2$]{And2010} in combination to Kim and Yee \cite[Theorem 1.4, eq.~(1.2) with $a=1$]{KY2013}:
\begin{align*}
	\sum_{N_1,\ldots, N_m\ge 0}\frac{q^{\sum_{i=1}^{m}\binom{N_i+1}{2}}}{(q;q)_{N_1}}\qbinom{N_1}{N_2}\qbinom{N_2}{N_3}\cdots \qbinom{N_{m-1}}{N_{m}} = \frac{ (q, q^{m+1}, q^{m+2};q^{m+2})_{\infty}}{ (q;q)_{\infty} (q;q^2)_{\infty}},
\end{align*}
which corresponds to the case $a=1$. 

\begin{remark}
    For other multisum representations for the infinite products in Theorem \ref{thm2.1}, see for instance, Chen, Sang and Shi \cite[Theorem 1.8, eq.~(1.7)]{chen} and Sang and Shi \cite[Theorem 1.11, eq.~(1.4)]{shi}. The main combinatorial objects considered in those two papers are overpartitions.
\end{remark}

\subsection{Proof of Theorem \ref{thm2.1}}

For the proof, there are three main ingredients:

\begin{enumerate}[label=\textup{(\arabic*)}, itemindent=*, leftmargin=*]
	\item Theorem \ref{th:sym-a}: \emph{A symmetry property} that allows us to reduce \eqref{eq:gen4} to the cases of $a\le \frac{m}{2} +1$;
	
	\item Theorem \ref{th:2a-general}: \emph{A $q$-binomial coefficient multisum transformation formula} that allows us to rewrite the summation side of \eqref{eq:gen4} as an alternative form;
	
	\item Theorems \ref{th:Andrews-even} and \ref{th:KY-odd}: \emph{Two identities of Andrews and Kim--Yee} that build a connection between the product side of \eqref{eq:gen4} and the above alternative summation side.
\end{enumerate}

For notational convenience, we map $a\mapsto m-a$ in \eqref{eq:gen4} and shall prove the following equivalent identity.

\begin{theorem}\label{th:gAG-original}
	For $m\ge 1$ and $0\le a\le m$,
	\begin{align}\label{eq:gAG-m>=2a}
		& \frac{ (q^{a+1},q^{m+1-a},q^{m+2};q^{m+2})_\infty}{(q;q)_\infty (q;q^2)_{\infty}} \notag \\ &\qquad =\sum_{N_1,\ldots, N_m\ge 0}\frac{q^{\sum_{i=1}^{m}\binom{N_i+1}{2}}}{(q;q)_{N_1}}\qbinom{N_1+\delta_{a,1}}{N_2}\qbinom{N_2+\delta_{a,2}}{N_3}\cdots \qbinom{N_{m-1}+\delta_{a,m-1}}{N_{m}}.
	\end{align}
\end{theorem}

\begin{proof}
	We first observe that the $a=0$ or $m$ cases of \eqref{eq:gAG-m>=2a} are exactly \eqref{eq:And-1.9} and \eqref{eq:KY-1.2} with $a=0$. Further, by the symmetry property in Theorem \ref{th:sym-a}, it suffices to show \eqref{eq:gAG-m>=2a} for $1\le a\le \frac{m}{2}$.
	Now in Theorem \ref{th:2a-general}, we choose
	\begin{align*}
		F(N_{2a})\mapsto \sum_{N_{2a+1},\ldots,N_{m}\ge 0}q^{\binom{N_{2a+1}+1}{2}+\cdots+\binom{N_m+1}{2}}\qbinom{N_{2a}}{N_{2a+1}}\cdots \qbinom{N_{m-1}}{N_m}.
	\end{align*}
	Then,
	\begin{align*}
		\RHS\eqref{eq:gAG-m>=2a}
		&=\sum_{N_1,\ldots,N_{m}\ge 0}\frac{q^{\sum_{i=1}^{m} \binom{N_i+1}{2}}}{(q;q)_{N_1}} \qbinom{N_1}{N_2}\cdots \qbinom{N_{a-1}}{N_{a}}  \qbinom{N_a+1}{N_{a+1}} \qbinom{N_{a+1}}{N_{a+2}}\cdots  \qbinom{N_{m-1}}{N_m}\\
		&=\sum_{N_1,\ldots,N_{m}\ge 0}\frac{q^{\sum_{i=1}^m \binom{N_i+1}{2}-\sum_{i=1}^a N_{2i} }}{(q;q)_{N_1}}\qbinom{N_1}{N_2}\cdots \qbinom{N_{m-1}}{N_{m}}\\
		&=\frac{(q^{a+1},q^{m+1-a},q^{m+2};q^{m+2})_\infty}{(q;q)_\infty(q;q^2)_{\infty}},
	\end{align*}
	where the second equality follows from Theorem~\ref{th:2a-general} and we make the use of \eqref{eq:And-1.9} and \eqref{eq:KY-1.2} for the last equality.
\end{proof}

\subsubsection{A symmetry property}

Notice that the product side of \eqref{eq:gAG-m>=2a} has some sort of symmetry in the sense that if we replace $a$ by $m-a$, the product stays invariant. We will show that the same symmetry property also holds true for the summation side.

\begin{theorem}\label{th:sym-a}
	For $m\ge 1$ and $1\le a\le m-1$,
	\begin{align}\label{eq:sym-a}
		&\sum_{N_1,\ldots,N_{m}\ge 0}\frac{q^{\sum_{i=1}^m \binom{N_i+1}{2}}}{(q;q)_{N_1}} \qbinom{N_1}{N_2}\cdots \qbinom{N_{a-1}}{N_{a}}  \qbinom{N_a+1}{N_{a+1}} \qbinom{N_{a+1}}{N_{a+2}}\cdots \qbinom{N_{m-1}}{N_{m}}\notag\\
		&= \sum_{N_1,\ldots,N_{m}\ge 0}\frac{q^{ \sum_{i=1}^m \binom{N_i+1}{2}}}{(q;q)_{N_1}} \qbinom{N_1}{N_2}\cdots \qbinom{N_{m-a-1}}{N_{m-a}}  \qbinom{N_{m-a}+1}{N_{m-a+1}} \qbinom{N_{m-a+1}}{N_{m-a+2}}\cdots \qbinom{N_{m-1}}{N_{m}}.
	\end{align}
\end{theorem}

The proof of this symmetry property is based on a key lemma given in Lemma \ref{le:1-2-a-i}.

\begin{proof}[Proof of Theorem \ref{th:sym-a}]
	First, \eqref{eq:sym-a} is trivial when $a=m-a$. Now, without loss of generality, we assume that $a\ge m-a+1$. By \eqref{binomtri1}, 
	$$
	\qbin{N_a+1}{N_{a+1}}=\qbinom{N_a}{N_{a+1}}+q^{N_a+1-N_{a+1}}\qbinom{N_a}{N_{a+1}-1}.
	$$
	Hence,
	\begin{align*}
		&\LHS\eqref{eq:sym-a}\\
		&=\sum_{N_1,\ldots,N_{m}\ge 0}\frac{q^{\sum_{i=1}^{m} \binom{N_{i}+1}{2}}}{(q;q)_{N_1}}
		\qbinom{N_1}{N_2}\cdots \qbinom{N_{a-1}}{N_{a}}  \qbinom{N_{a}+1}{N_{a+1}} \qbinom{N_{a+1}}{N_{a+2}}\cdots \qbinom{N_{m-1}}{N_{m}}\\
		&=\sum_{N_1,\ldots,N_{m}\ge 0}\frac{q^{\binom{N_1+1}{2}+\cdots+\binom{N_{m}+1}{2}}}{(q;q)_{N_1}}\qbinom{N_1}{N_2}\cdots \qbinom{N_{m-1}}{N_{m}}\\
		&\quad+\sum_{N_1,\ldots,N_{m}\ge 0}\frac{q^{ \sum_{i=1}^{m} \binom{N_i+1}{2} + (N_{a}+1)}}{(q;q)_{N_1}}
		\qbinom{N_1}{N_2}\cdots \qbinom{N_{a}}{N_{a+1}} \qbinom{N_{a+1}+1}{N_{a+2}} \qbinom{N_{a+2}}{N_{a+3}}  \cdots \qbinom{N_{m-1}}{N_{m}},
		\end{align*}
		where $N_{a+1}$ is replaced by $N_{a+1}+1$ to get the second sum in the last line. 
If we apply \eqref{binomtri1} to ${{N_{a+1}+1}\brack {N_{a+2}}}$ in the second sum above, then
\begin{align*}
&\LHS\eqref{eq:sym-a}\\		
		&=\sum_{N_1,\ldots,N_{m}\ge 0}\frac{q^{\binom{N_1+1}{2}+\cdots+\binom{N_{m}+1}{2}}}{(q;q)_{N_1}}\qbinom{N_1}{N_2}\cdots \qbinom{N_{m-1}}{N_{m}}\\
		 \\
		&\quad +\sum_{N_1,\ldots, N_m} \frac{q^{\sum_{i=1}^{m} \binom{N_i+1}{2} + (N_{a}+1)}}{(q;q)_{N_1}}
		\qbinom{N_{1}}{N_{2}} \cdots \qbinom{N_{m-1}}{N_{m}}\\
		&\quad+\sum_{N_1,\ldots,N_{m}\ge 0}\frac{q^{\sum_{i=1}^{m} \binom{N_i+1}{2}+ (N_a+1)+(N_{a+1}+1) }}{(q;q)_{N_1}}
		 \qbinom{N_1}{N_2}\cdots \qbinom{N_{a+1}}{N_{a+2}}  \qbinom{N_{a+2}+1}{N_{a+3}} \qbinom{N_{a+3}}{N_{a+4}} \cdots \qbinom{N_{m-1}}{N_{m}}.
	\end{align*}
Upon repeating this process, we are eventually led to
	\begin{align*}
		\LHS\eqref{eq:sym-a}
		&= \sum_{N_1,\ldots,N_{m}\ge 0}\frac{q^{\sum_{i=1}^{m} \binom{N_i+1}{2} }}{(q;q)_{N_1}}\qbinom{N_1}{N_2}\cdots \qbinom{N_{m-1}}{N_{m}}\notag\\
		&+ \sum_{j=a}^{m-1} \left(\sum_{N_1,\ldots N_{m}\ge 0}\frac{q^{\sum_{i=1}^{m} \binom{N_i+1}{2}+ \sum_{i=a}^{j} (N_{i}+1) }}{(q;q)_{N_1}}\qbinom{N_1}{N_2}\cdots \qbinom{N_{m-1}}{N_{m}} \right).
	\end{align*} 
	Replacing $a$ with $m-a$ also gives
	\begin{align*}
		\RHS\eqref{eq:sym-a}
		&=\sum_{N_1,\ldots,N_{m}\ge 0}\frac{q^{\sum_{i=1}^{m} \binom{N_i+1}{2}}}{(q;q)_{N_1}}\qbinom{N_1}{N_2}\cdots \qbinom{N_{m-1}}{N_{m}}\notag\\
		&+ \sum_{j=m-a}^{m-1} \left( \sum_{N_1,\ldots N_{m}\ge 0}\frac{q^{\sum_{i=1}^{m} \binom{N_i+1}{2}+\sum_{i=m-a}^{j}  (N_{i}+1) }}{(q;q)_{N_1}}\qbinom{N_1}{N_2}\cdots \qbinom{N_{m-1}}{N_{m}}\right).
	\end{align*} 
	In Lemma \ref{le:1-2-a-i} (ii), taking $(m,a,i)\mapsto (1,a-1,m-a)$, choosing $\uk=\big((1)_{m}\big)$, and fixing $F(N)\equiv 1$ for all $N\ge 0$, we shall see that $\LHS\eqref{eq:sym-a}=\RHS\eqref{eq:sym-a}$.
\end{proof}

\subsubsection{A $q$-binomial coefficient multisum transformation formula}

For $1\le a\le m-1$, exactly one $q$-binomial coefficient on the summation side of \eqref{eq:gAG-m>=2a} is of the form ${N_i+1 \brack N_{i+1}}$. Our next result allows us to get rid of the annoying ``plus one'' in $N_i+1$.

\begin{theorem}\label{th:2a-general}
	Let $\{F(N)\}_{N\ge 0}$ be a family of series in $q$. For $a\ge 1$,
	\begin{align}\label{eq:2a-general}
		&\sum_{N_1,\ldots,N_{2a}\ge 0}\frac{q^{ \sum_{i=1}^ { 2a } \binom{N_i+1}{2}} F(N_{2a})}{(q;q)_{N_1}} \qbinom{N_1}{N_2}\cdots \qbinom{N_{a-1}}{N_{a}}  \qbinom{N_a+1}{N_{a+1}} \qbinom{N_{a+1}}{N_{a+2}}\cdots \qbinom{N_{2a-1}}{N_{2a}}\notag\\
		&= \sum_{N_1,\ldots,N_{2a}\ge 0}\frac{q^{ \sum_{i=1}^ { 2a}  \binom{N_i+1}{2} -\sum_{i=1}^a N_{2i}} F(N_{2a})}{(q;q)_{N_1}}\qbinom{N_1}{N_2}\qbinom{N_2}{N_3}\cdots \qbinom{N_{2a-1}}{N_{2a}}.
	\end{align}
\end{theorem}

We still use Lemma \ref{le:1-2-a-i} to establish Theorem \ref{th:2a-general}. However, since the proof of this lemma is far too lengthy, it will be presented after the preparation of all necessary lemmas.

\subsubsection{Identities of Andrews and Kim--Yee}

The product side of \eqref{th:gAG-original} in connection to a certain multisum has already appeared in the work of Andrews \cite{And2010} and Kim and Yee \cite{KY2013}. In particular, Andrews considered the $m$ even case and Kim and Yee studied the $m$ odd case. We record an equivalent form of their results.

\begin{theorem}[Andrews]\label{th:Andrews-even}
	For $m\ge 1$ and $a\ge 0$ with $m\ge a$,
	\begin{align}\label{eq:And-1.9}
	    &\frac{(q^{a+1},q^{2m+1-a},q^{2m+2};q^{2m+2})_\infty}{(q;q)_\infty(q;q^2)_{\infty}}\notag\\
		&\qquad=\sum_{N_1,\ldots,N_{2m}\ge 0}\frac{q^{\sum_{i=1}^m \binom{N_i+1}{2} - \sum_{i=1}^a N_{2i} }}{(q;q)_{N_1}} \qbinom{N_1}{N_2}\qbinom{N_2}{N_3}\cdots \qbinom{N_{2m-1}}{N_{2m}}.
	\end{align}
\end{theorem}

\begin{proof}
	Let $N_i=n_i+n_{i+1}+\cdots +n_{k-1}$. Recall that \cite[Theorem 3, eq.~(1.9)]{And2010} tells us that for $k, b\ge 2$ and $k\ge b-1$ with $k$ odd and $b$ even,
	\begin{align}
	    &\frac{(-q^2;q^2)_\infty (q^{b},q^{2k+2-b},q^{2k+2};q^{2k+2})_\infty}{(q^2;q^2)_\infty}\notag\\
		&\qquad=\sum_{n_1,\ldots,n_{k-1}\ge 0}\frac{q^{\sum_{i=1}^{k-1}N_i^2+n_1+n_3+\cdots+n_{b-3}+N_{b-1}+N_{b}+\cdots +N_{k-1}}}{(q^2;q^2)_{n_1}(q^2;q^2)_{n_2}\cdots (q^2;q^2)_{n_{k-1}}}.
	\end{align}
	Although Andrews only stated this result for $k\ge b$ in \cite[Theorem 3]{And2010}, the case of $b=k+1$ had also been shown by him by making use of \cite[(4.28) with $A=2k+2$]{And2010}, and then recalling \cite[(4.5) and (4.27) for $x=1$]{And2010} with \cite[(2.8)]{And2010} invoked.
	
	Let us choose $(b,k)\mapsto (2a+2,2m+1)$. Since $k\ge b-1$, we have $m\ge a$. Now we replace $q^2$ with $q$. Noticing that $n_i=N_i-N_{i+1}$ for $1\le i\le 2m-1$ and $n_{2m}=N_{2m}$, we have
	\begin{align*}
		&\frac{(-q;q)_\infty (q^{a+1},q^{2m+1-a},q^{2m+2};q^{2m+2})_\infty}{(q;q)_\infty}\\
		&=\sum_{n_1,\ldots,n_{2m}\ge 0}\frac{q^{\sum_{i=1}^{2m}\frac{N_i^2}{2}+\frac{1}{2}(n_1+n_3+\cdots+n_{2a-1}+N_{2a+1}+N_{2a+2}+\cdots +N_{2m})}}{(q;q)_{n_1}(q;q)_{n_2}\cdots (q;q)_{n_{2m}}}\\
		&=\sum_{n_1,\ldots,n_{2m}\ge 0}\frac{q^{\sum_{i=1}^{2m}\frac{N_i^2}{2}+\frac{1}{2}(N_1-N_2+N_3-N_4+\cdots+N_{2a-1}-N_{2a}+N_{2a+1}+N_{2a+2}+\cdots +N_{2m})}}{(q;q)_{N_1-N_2}(q;q)_{N_2-N_3}\cdots (q;q)_{N_{2m}}}\\
		&=\sum_{N_1,\ldots,N_{2m}\ge 0}\frac{q^{\binom{N_1+1}{2}+\binom{N_2}{2}+\cdots+\binom{N_{2a-1}+1}{2}+\binom{N_{2a}}{2}+\binom{N_{2a+1}+1}{2}+\cdots+\binom{N_{2m}+1}{2}}}{(q;q)_{N_1}}\qbinom{N_1}{N_2}\cdots \qbinom{N_{2m-1}}{N_{2m}},
	\end{align*}
	confirming the desired result.
\end{proof}

\begin{theorem}[Kim--Yee]\label{th:KY-odd}
	For $m\ge 1$ and $a\ge 0$ with $m\ge a+1$,
	\begin{align}\label{eq:KY-1.2}
	    &\frac{(q^{a+1},q^{2m-a},q^{2m+1};q^{2m+1})_\infty}{(q;q)_\infty(q;q^2)_{\infty}}\notag\\
		&\qquad=\sum_{N_1,\ldots,N_{2m-1}\ge 0}\frac{q^{ \sum_{i=1}^m \binom{N_i+1}{2} -\sum_{i=1}^a N_{2i} }}{(q;q)_{N_1}}\qbinom{N_1}{N_2}\qbinom{N_2}{N_3}\cdots \qbinom{N_{2m-2}}{N_{2m-1}}.
	\end{align}
\end{theorem}

\begin{proof}
	Let $N_i=n_i+n_{i+1}+\cdots +n_{k-1}$. We know from \cite[Theorem 1.4, eq.~(1.2)]{KY2013} that for $k\ge b\ge 1$ with $k$ even and $b$ odd,
	\begin{align}
	    &\frac{(-q^2;q^2)_\infty (q^{b+1},q^{2k+1-b},q^{2k+2};q^{2k+2})_\infty}{(q^2;q^2)_\infty}\notag\\
		&\qquad=\sum_{n_1,\ldots,n_{k-1}\ge 0}\frac{q^{\sum_{i=1}^{k-1}N_i^2+n_1+n_3+\cdots+n_{b-2}+N_{b}+N_{b+1}+\cdots +N_{k-1}}}{(q^2;q^2)_{n_1}(q^2;q^2)_{n_2}\cdots (q^2;q^2)_{n_{k-1}}}.
	\end{align}
	Choosing $(b,k)\mapsto (2a+1,2m)$ and replacing $q^2$ with $q$, the desired result follows.
\end{proof}

\subsection{A decomposition formula for $q$-binomial coefficient multisums}

Now let us start with a decomposition formula that lays the foundation for our key lemmas to be established in the next section.

\begin{lemma}\label{le:q-binomial-1-2}
	Let $\{F(N)\}_{N\ge 0}$ be a family of series in $q$. Given an arbitrary integer $a\ge 2$, we have that, for any nonnegative integers $k_1,\ldots,k_{a-1}$,
	\begin{align}
		&\sum_{N_1,\ldots N_a\ge 0}\frac{q^{\sum_{i=1}^{a-2} \binom{N_i+k_i}{2}+\binom{N_{a-1}+k_{a-1}}{2}+\binom{N_{a}+k_{a-1}+1}{2}}F(N_a)}{(q;q)_{N_1}}\qbinom{N_1}{N_2}\qbinom{N_2}{N_3}\cdots \qbinom{N_{a-1}}{N_a}\notag\\
		&=\sum_{N_1,\ldots N_a\ge 0}\frac{q^{\sum_{i=1}^{a-2} \binom{N_i+k_i}{2}+\binom{N_{a-1}+k_{a-1}+1}{2}+\binom{N_{a}+k_{a-1}}{2}}F(N_a)}{(q;q)_{N_1}}\qbinom{N_1}{N_2}\qbinom{N_2}{N_3}\cdots \qbinom{N_{a-1}}{N_a}\notag\\
		&+\sum_{N_1,\ldots N_a\ge 0}\frac{q^{\sum_{i=1}^{a-2} \binom{N_i+k_i+1}{2}+\binom{N_{a-1}+k_{a-1}+1}{2}+\binom{N_{a}+k_{a-1}+1}{2}}F(N_a)}{(q;q)_{N_1}}\qbinom{N_1}{N_2}\qbinom{N_2}{N_3}\cdots \qbinom{N_{a-1}}{N_a}. \label{eq:q-binomial-1-2}
	\end{align}
\end{lemma}

\begin{proof}
First, by replacing $N_i$ by $N_i-1$ for $1\le i\le a-1$,  we can rewrite the second sum on the right hand side of \eqref{eq:q-binomial-1-2} as follows:
\begin{align*}
& \sum_{\substack{N_1,\ldots , N_{a-1} \ge 1,\\ N_a\ge 0} }\frac{q^{\sum_{i=1}^{a-2} \binom{N_i+k_i }{2} +\binom{N_{a-1}+k_{a-1}}{2}+\binom{N_{a}+k_{a-1}+1}{2}}F(N_a)}{(q;q)_{N_1}}\qbinom{N_1}{N_2} \cdots \qbinom{N_{a-2}}{N_{a-1}} \qbinom{N_{a-1}}{N_a} (1-q^{N_{a-1}-N_a})\\
&=\sum_{{N_1,\ldots, N_a\ge 0} }\frac{q^{\sum_{i=1}^{a-2} \binom{N_i+k_i }{2} +\binom{N_{a-1}+k_{a-1}}{2}+\binom{N_{a}+k_{a-1}+1}{2}}F(N_a)}{(q;q)_{N_1}}\qbinom{N_1}{N_2} \cdots \qbinom{N_{a-2}}{N_{a-1}} \qbinom{N_{a-1} -1}{N_a } (1-q^{N_{a-1}}),
\end{align*}
where the equality follows from \eqref{eq:q-binomial-trivial}.
 Thus, to prove \eqref{eq:q-binomial-1-2},  it is sufficient to prove that the coefficients of $F(N_a)$ on both sides match, i.e., for fixed $N_{a-1}$ and $N_a$, we need to show
\begin{align}
&  q^{N_{a}}   \qbinom{N_{a-1}}{N_a}
=  q^{N_{a-1}}   \qbinom{N_{a-1}}{N_a} 
 +  q^{ N_{a}}  \qbinom{N_{a-1}-1}{N_a} (1-q^{N_{a-1}}). \label{eq:pf-lem-q-binomial-1-2}
\end{align}
This follows from the definition of the $q$-binomial coefficients.  
\end{proof}

\subsection{Key lemmas}

For notational convenience, we write for $(k_1,k_2,\ldots,k_b)$ a finite list of nonnegative integers and $\{F(N)\}_{N\ge 0}$ a family of series in $q$,
\begin{align}\label{eq:S-sum-def}
	\fS\begin{pmatrix}
		k_1\\ k_2\\ \vdots\\ k_b
	\end{pmatrix}:=\fS_F\begin{pmatrix}
		k_1\\ k_2\\ \vdots\\ k_b
	\end{pmatrix}:=\sum_{N_1,\ldots N_b\ge 0}\frac{q^{\sum_{i=1}^b\binom{N_i+k_i}{2}}F(N_b)}{(q;q)_{N_1}}\qbinom{N_1}{N_2}\qbinom{N_2}{N_3}\cdots \qbinom{N_{b-1}}{N_b}.
\end{align}
For each $k_i$, we \emph{always} maintain the subscript $i$ to keep track of its corresponding summation index, i.e., the $\binom{N_i+k_i}{2}$ term in the power of $q$ in \eqref{eq:S-sum-def}. In particular, we write $k_i=(m)_i$ if $m$ is a specific value (e.g., $(k_{a-1}+1)_a$ in \eqref{eq:S-1-2} as $k_{a-1}+1$ is a specific value which corresponds to the summation index $a$, i.e., the $\binom{N_a+k_{a-1}+2}{2}$ in the power of $q$).

Notice that in Lemma \ref{le:q-binomial-1-2}, $\{F(N)\}_{N\ge 0}$ is an arbitrary family of series. So, by abuse of notation, if we take
\begin{align*}
	F(N_a)\mapsto \sum_{N_{a+1},\ldots N_{b}\ge 0}q^{\sum_{i=a+1}^b\binom{N_i+k_i}{2}}F(N_b)\qbinom{N_a}{N_{a+1}}\cdots \qbinom{N_{b-1}}{N_b},
\end{align*}
then Lemma \ref{le:q-binomial-1-2} implies the following result.

\begin{corollary}
	For $b\ge a\ge 2$,
	\begin{align}\label{eq:S-1-2}
		\fS\mbox{\scalebox{0.8}{$\begin{pmatrix}
					k_1\\ \vdots\\ k_{a-2}\\ k_{a-1}\\ (k_{a-1}+1)_a\\ k_{a+1}\\ \vdots\\ k_b
				\end{pmatrix}$}}
		=\fS\mbox{\scalebox{0.8}{$\begin{pmatrix}k_1\\ \vdots\\ k_{a-2}\\ (k_{a-1}+1)_{a-1}\\ (k_{a-1})_a\\ k_{a+1}\\ \vdots\\ k_b\end{pmatrix}$}}
		+\fS\mbox{\scalebox{0.8}{$\begin{pmatrix}(k_1+1)_1\\ \vdots\\ (k_{a-2}+1)_{a-2}\\ (k_{a-1}+1)_{a-1}\\ (k_{a-1}+1)_a \;\; \;\;\;  \\ k_{a+1}\\ \vdots\\ k_b\end{pmatrix}$}}.
	\end{align}
\end{corollary}

The above relation is a key to the next two lemmas.

\begin{lemma}\label{le:1-2-a-1}
	For $a\ge 1$, $m\ge 0$ and $\uk$ a finite list of nonnegative integers,
	\begin{align}\label{eq:1-2-a-1}
		\fS\mbox{\scalebox{0.8}{$\begin{pmatrix}(m)_1\\ \vdots\\(m)_{a-1}\\ (m)_a\\ (m+1)_{a+1}\\ \uk\end{pmatrix}$}}
		-\fS\mbox{\scalebox{0.8}{$\begin{pmatrix}(m)_1\\ \vdots\\ (m)_{a-1}\\ (m+1)_{a}\\ (m)_{a+1}\\ \uk\end{pmatrix}$}}
		=\fS\mbox{\scalebox{0.8}{$\begin{pmatrix}(m+1)_1\\ \vdots\\(m+1)_{a-1}\\(m+1)_{a}\\ (m+1)_{a+1}\\ \uk\end{pmatrix}$}}.
	\end{align}
\end{lemma}

\begin{proof}
	This is a direct consequence of \eqref{eq:S-1-2} by taking $k_1=\cdots = k_{a-1}=m$ and $(k_{a+1},\ldots,k_b)=\uk$, and then replacing $a$ by $a+1$.
\end{proof}

\begin{lemma}\label{le:1-2-a-i}
	Let $\uk$ be a finite list of nonnegative integers. The following are true.
	\begin{enumerate}[label=\textup{(\roman*)}, widest=ii, itemindent=*, leftmargin=*]
		\item For $a\ge i\ge 2$ and $m\ge 0$,
		\begin{align}\label{eq:1-2-a-i}
			\fS\mbox{\scalebox{0.8}{$\begin{pmatrix}(m)_1\\ \vdots\\ (m)_a\\ (m+1)_{a+1}\\ \vdots\\ (m+1)_{a+i-1}\\ (m+1)_{a+i}\\ \uk\end{pmatrix}$}}
			-\fS\mbox{\scalebox{0.8}{$\begin{pmatrix}(m)_1\\ \vdots\\ (m)_{a-1}\\ (m+1)_{a}\\ \vdots\\ (m+1)_{a+i-1}\\ (m)_{a+i}\\ \uk\end{pmatrix}$}}
			=\fS\mbox{\scalebox{0.8}{$\begin{pmatrix}(m)_1\\ \vdots\\ (m)_{i-1}\\ (m+1)_{i}\\ \vdots\\ (m+1)_{a+i-1}\\ (m+1)_{a+i}\\ \uk\end{pmatrix}$}}
			-\fS\mbox{\scalebox{0.8}{$\begin{pmatrix}(m)_1\\ \vdots\\ (m)_{i-2}\\ (m+1)_{i-1}\\ \vdots\\ (m+1)_{a+i-1}\\ (m)_{a+i}\\ \uk\end{pmatrix}$}}.
		\end{align}
		
		\item For $a\ge i\ge 1$ and $m\ge 0$,
		\begin{align}\label{eq:S-i-a+1}
			\fS\mbox{\scalebox{0.8}{$\begin{pmatrix}(m)_1\\ \vdots\\ (m)_a\\ (m+1)_{a+1}\\ (m)_{a+2}\\ \vdots\\ (m)_{a+i}\\ \uk\end{pmatrix}$}}
			+\cdots
			+\fS\mbox{\scalebox{0.8}{$\begin{pmatrix}(m)_1\\ \vdots\\ (m)_a\\ (m+1)_{a+1}\\ (m+1)_{a+2}\\ \vdots\\ (m+1)_{a+i}\\ \uk\end{pmatrix}$}}=
			\fS\mbox{\scalebox{0.8}{$\begin{pmatrix}(m)_1\\ \vdots\\ (m)_{i-1}\\ (m+1)_{i}\\ (m)_{i+1}\\ \vdots\\ (m)_{a+i}\\ \uk\end{pmatrix}$}}
			+\cdots
			+\fS\mbox{\scalebox{0.8}{$\begin{pmatrix}(m)_1\\ \vdots\\ (m)_{i-1}\\ (m+1)_{i}\\ (m+1)_{i+1}\\ \vdots\\ (m+1)_{a+i}\\ \uk\end{pmatrix}$}}.
		\end{align}
	\end{enumerate}
\end{lemma}

\begin{proof}
	We prove this lemma by induction on $a$. The basic idea can be outlined as follows.
	\begin{enumerate}[label=(\arabic*), itemindent=*, leftmargin=*]
		\item We start with the base case, $a=1$, of Part (ii).
		
		\item Next, assuming that Part (ii) is true for some $a=b\ge 1$, we will show that Part (i) is true for $a=b+1$.
		
		\item Finally, we use Part (i) for $a=b+1$ and Part (ii) for $a=b$ to show Part (ii) for $a=b+1$, and therefore complete the proof.
	\end{enumerate}
	
	\textit{Step (1).} The $a=1$ (and thus $i=1$) case of Part (ii) is
	\begin{align*}
		\fS\begin{pmatrix}
			(m)_1 \\ (m+1)_2\\ \uk
		\end{pmatrix}
		=\fS\begin{pmatrix}
			(m+1)_1 \\ (m)_2\\ \uk
		\end{pmatrix}
		+\fS\begin{pmatrix}
			(m+1)_1 \\ (m+1)_2\\ \uk
		\end{pmatrix},
	\end{align*}
	which is a direct consequence of \eqref{eq:S-1-2}.
	
	\textit{Step (2).} Assume that Part (ii) is true for some $a=b\ge 1$. We will show that for any $j$ with $2\le j\le b+1$,
	\begin{align}\label{eq:1-2-b+1-i}
		\fS\mbox{\scalebox{0.8}{$\begin{pmatrix}(m)_1\\ \vdots\\ (m)_{b} \\ \;\;\;\; (m)_{b+1}\\ (m+1)_{b+2}\\ \vdots\\
		(m+1)_{b+j}\\(m+1)_{b+j+1}\\ \uk\end{pmatrix}$}}
		-\fS\mbox{\scalebox{0.8}{$\begin{pmatrix}(m)_1\\ \vdots\\ (m)_{b}\\ (m+1)_{b+1}\\ (m+1)_{b+2} \\ \vdots\\ (m+1)_{b+j}\\ (m)_{b+j+1}\\ \uk\end{pmatrix}$}}
		=\fS\mbox{\scalebox{0.8}{$\begin{pmatrix}(m)_1\\ \vdots\\ (m)_{j-2} \\ (m)_{j-1}\\ (m+1)_{j}\\ \vdots\\ (m+1)_{b+j}\\ (m+1)_{b+j+1}\\ \uk\end{pmatrix}$}}
		-\fS\mbox{\scalebox{0.8}{$\begin{pmatrix}(m)_1\\ \vdots\\ (m)_{j-2}\\  (m+1)_{j-1}\\(m+1)_{j} \;\;\;\;  \\ \vdots
		\\  (m+1)_{b+j}\\  (m)_{b+j+1}\\ \uk\end{pmatrix}$}}.
	\end{align}
	We repeatedly make use of \eqref{eq:S-1-2} to obtain
	\begin{align*}
		\fS\mbox{\scalebox{0.8}{$\begin{pmatrix}(m)_1\\ \vdots\\ (m)_{b} \\ (m)_{b+1}\\ (m+1)_{b+2}\\  (m+1)_{b+3}\\ \vdots\\ (m+1)_{b+j+1}\\ \uk\end{pmatrix}$}}
		&=\fS\mbox{\scalebox{0.8}{$\begin{pmatrix}(m)_1\\ \vdots\\ (m)_{b}\\ (m+1)_{b+1}\\ (m)_{b+2}\\ (m+1)_{b+3}\\ \vdots\\ (m+1)_{b+j+1}\\ \uk\end{pmatrix}$}}
		+\fS\mbox{\scalebox{0.8}{$\begin{pmatrix}(m+1)_1\\ \vdots\\ (m+1)_{b+j+1}\\ \uk\end{pmatrix}$}}\\
		&=\fS\mbox{\scalebox{0.8}{$\begin{pmatrix}(m)_1\\ \vdots\\ (m)_{b}\\ (m+1)_{b+1}\\ (m+1)_{b+2}\\ (m)_{b+3}\\ (m+1)_{b+4}\\ \vdots\\ (m+1)_{b+j+1}\\ \uk\end{pmatrix}$}}
		+\fS\mbox{\scalebox{0.8}{$\begin{pmatrix}(m+1)_1\\ \vdots\\ (m+1)_{b}\\ (m+2)_{b+1}\\ (m+1)_{b+2}  \\ (m+1)_{b+3}  \\ (m+1)_{b+4} \\ \vdots\\ (m+1)_{b+j+1}\\ \uk\end{pmatrix}$}}
		+\fS\mbox{\scalebox{0.8}{$\begin{pmatrix}(m+1)_1\\ \vdots\\ (m+1)_{b+j+1}\\ \uk\end{pmatrix}$}}\\
		&= \cdots\\
		&=\fS\mbox{\scalebox{0.8}{$\begin{pmatrix}(m)_1\\ \vdots\\ (m)_{b}\\ (m+1)_{b+1}\\ \vdots\\ (m+1)_{b+j-1}\\  (m+1)_{b+j}\\ (m)_{b+j+1}\\ \uk\end{pmatrix}$}}
		+\fS\mbox{\scalebox{0.8}{$\begin{pmatrix}(m+1)_1\\ \vdots\\ (m+1)_{b}\\ (m+2)_{b+1}\\ \vdots\\ (m+2)_{b+j-1}\\ (m+1)_{b+j}\\ (m+1)_{b+j+1}\\ \uk\end{pmatrix}$}}
		+\cdots
		+\fS\mbox{\scalebox{0.8}{$\begin{pmatrix}(m+1)_1\\ \vdots\\ (m+1)_{b}\\ (m+2)_{b+1}\\ (m+1)_{b+2}\\ \vdots \\ (m+1)_{b+j}\\ (m+1)_{b+j+1}\\ \uk\end{pmatrix}$}}
		+\fS\mbox{\scalebox{0.8}{$\begin{pmatrix}(m+1)_1\\ \vdots\\ (m+1)_{b+j+1}\\ \uk\end{pmatrix}$}}.
	\end{align*}
	Thus,
	\begin{align*}
		\LHS\eqref{eq:1-2-b+1-i}
		&=\fS\mbox{\scalebox{0.8}{$\begin{pmatrix}(m+1)_1\\ \vdots\\ (m+1)_{b}\\ (m+2)_{b+1}\\ \vdots\\ (m+2)_{b+j-1}\\ (m+1)_{b+j}\\ (m+1)_{b+j+1}\\ \uk\end{pmatrix}$}}
		+\cdots
		+\fS\mbox{\scalebox{0.8}{$\begin{pmatrix}(m+1)_1\\ \vdots\\ (m+1)_{b}\\ (m+2)_{b+1}\\ (m+1)_{b+2}\\ \vdots\\ (m+1)_{b+j}\\  (m+1)_{b+j+1}\\ \uk\end{pmatrix}$}}
		+\fS\mbox{\scalebox{0.8}{$\begin{pmatrix}(m+1)_1\\ \vdots\\ (m+1)_{b+j+1}\\ \uk\end{pmatrix}$}}.
	\end{align*}
	A similar application of \eqref{eq:S-1-2} to the right hand side of \eqref{eq:S-i-a+1} gives
	\begin{align*}
		\RHS\eqref{eq:1-2-b+1-i}
		&=\fS\mbox{\scalebox{0.8}{$\begin{pmatrix}(m+1)_1\\ \vdots\\ (m+1)_{j-2}\\ (m+2)_{j-1}\\ \vdots\\ (m+2)_{b+j-1}\\ (m+1)_{b+j}\\ (m+1)_{b+j+1}\\ \uk\end{pmatrix}$}}
		+\cdots
		+\fS\mbox{\scalebox{0.8}{$\begin{pmatrix}(m+1)_1\\ \vdots\\ (m+1)_{j-2}\\ (m+2)_{j-1}\\ (m+1)_{j}\\ \vdots\\   (m+1)_{b+j}\\  (m+1)_{b+j+1}\\ \uk\end{pmatrix}$}}
		+\fS\mbox{\scalebox{0.8}{$\begin{pmatrix}(m+1)_1\\ \vdots\\ (m+1)_{b+j+1}\\ \uk\end{pmatrix}$}}.
	\end{align*}
	Now, in \eqref{eq:S-i-a+1}, we first take $(m,a,i)\mapsto (m+1,b,j-1)$ and then with abuse of notation choose $\uk\mapsto \big((m+1)_{b+j},(m+1)_{b+j+1},\uk\big)$. It follows that $\LHS\eqref{eq:1-2-b+1-i}=\RHS\eqref{eq:1-2-b+1-i}$, and therefore \eqref{eq:1-2-b+1-i} holds true.
	
	\textit{Step (3).} We want to show Part (ii) for $a=b+1$. That is, for $1\le i\le b+1$,
	\begin{align}\label{eq:S-i-b}
		\fS\mbox{\scalebox{0.8}{$\begin{pmatrix}(m)_1\\ \vdots\\ (m)_{b+1}\\ (m+1)_{b+2}\\ (m)_{b+3}\\ \vdots\\ (m)_{b+i+1}\\ \uk\end{pmatrix}$}}
		+\cdots
		+\fS\mbox{\scalebox{0.8}{$\begin{pmatrix}(m)_1\\ \vdots\\ (m)_{b+1}\\ (m+1)_{b+2}\\ (m+1)_{b+3}\\  \vdots\\ (m+1)_{b+i+1}\\ \uk\end{pmatrix}$}}
		=
		\fS\mbox{\scalebox{0.8}{$\begin{pmatrix}(m)_1\\ \vdots\\ (m)_{i-1}\\ (m+1)_{i}\\ (m)_{i+1}\\ \vdots\\ (m)_{b+i+1}\\ \uk\end{pmatrix}$}}
		+\cdots
		+\fS\mbox{\scalebox{0.8}{$\begin{pmatrix}(m)_1\\ \vdots\\ (m)_{i-1}\\ (m+1)_{i} \\ (m+1)_{i+1}\\ \vdots\\  (m+1)_{b+i+1}\\ \uk\end{pmatrix}$}}.
	\end{align}
	First, in \eqref{eq:1-2-a-1}, we take $a\mapsto b+1$, and then with abuse of notation choose $\uk=\big((m)_{b+3},\uk\big)$. Thus,
	\begin{align*}
		\fS\mbox{\scalebox{0.8}{$\begin{pmatrix}(m)_1\\ \vdots\\ (m)_b \\ (m)_{b+1}\\ (m+1)_{b+2}\\ (m)_{b+3}\\ \uk\end{pmatrix}$}}
		-\fS\mbox{\scalebox{0.8}{$\begin{pmatrix}(m)_1\\ \vdots\\ (m)_{b}\\ (m+1)_{b+1}\\ (m)_{b+2}\\ (m)_{b+3}\\ \uk\end{pmatrix}$}}
		=\fS\mbox{\scalebox{0.8}{$\begin{pmatrix}(m+1)_1\\ \vdots\\ \vdots \\ (m+1)_{b+1} \\ (m+1)_{b+2}\\ (m)_{b+3}\\ \uk\end{pmatrix}$}}.
	\end{align*}
	Adding this with the $j=2$ case of \eqref{eq:1-2-b+1-i}, i.e.,
	\begin{align*}
		\fS\mbox{\scalebox{0.8}{$\begin{pmatrix}(m)_1\\ \vdots\\ (m)_b\\ (m)_{b+1}\\ (m+1)_{b+2}\\ (m+1)_{b+3}\\ \uk\end{pmatrix}$}}
		-\fS\mbox{\scalebox{0.8}{$\begin{pmatrix}(m)_1\\ \vdots\\ (m)_{b}\\ (m+1)_{b+1}\\ (m+1)_{b+2}\\ (m)_{b+3}\\ \uk\end{pmatrix}$}}
		=\fS\mbox{\scalebox{0.8}{$\begin{pmatrix}(m)_{1}\\ (m+1)_{2}\\ \vdots\\ \vdots  \\ (m+1)_{b+2} \\ (m+1)_{b+3}\\ \uk\end{pmatrix}$}}
		-\fS\mbox{\scalebox{0.8}{$\begin{pmatrix}(m+1)_{1}\\ \vdots\\ \vdots  \\ (m+1)_{b+1} \\ (m+1)_{b+2}\\ (m)_{b+3}\\ \uk\end{pmatrix}$}},
	\end{align*}
	we have
	\begin{align*}
		\fS\mbox{\scalebox{0.8}{$\begin{pmatrix}(m)_1\\ \vdots\\ (m)_b \\ (m)_{b+1}\\ (m+1)_{b+2}\\ (m)_{b+3}\\ \uk\end{pmatrix}$}}
		+\fS\mbox{\scalebox{0.8}{$\begin{pmatrix}(m)_1\\ \vdots\\ (m)_b \\ (m)_{b+1}\\ (m+1)_{b+2}\\ (m+1)_{b+3}\\ \uk\end{pmatrix}$}}
		&=\fS\mbox{\scalebox{0.8}{$\begin{pmatrix}(m)_1\\ \vdots\\ (m)_{b}\\ (m+1)_{b+1}\\ (m)_{b+2}\\ (m)_{b+3}\\ \uk\end{pmatrix}$}}
		+\fS\mbox{\scalebox{0.8}{$\begin{pmatrix}(m)_1\\ \vdots\\ (m)_{b}\\ (m+1)_{b+1}\\ (m+1)_{b+2}\\ (m)_{b+3}\\ \uk\end{pmatrix}$}}
		+\fS\mbox{\scalebox{0.8}{$\begin{pmatrix}(m)_{1}\\ (m+1)_{2}\\ \vdots\\ \vdots  \\ (m+1)_{b+2} \\ (m+1)_{b+3}\\ \uk\end{pmatrix}$}}.
	\end{align*}
	We then choose $\uk=\big((m)_{b+4},\uk\big)$ in the above and then add it with the $j=3$ case of \eqref{eq:1-2-b+1-i}. Further, if we repeat this procedure until adding the $j=i$ case of \eqref{eq:1-2-b+1-i}, it follows that, for $1\le i\le b+1$,
	\begin{align*}
		&\fS\mbox{\scalebox{0.8}{$\begin{pmatrix}(m)_1\\ \vdots\\ (m)_{b+1}\\ (m+1)_{b+2}\\ (m)_{b+3}\\ \vdots\\ (m)_{b+i+1}\\ \uk\end{pmatrix}$}}
		+\cdots
		+\fS\mbox{\scalebox{0.8}{$\begin{pmatrix}(m)_1\\ \vdots\\ (m)_{b+1}\\ (m+1)_{b+2}\\ (m+1)_{b+3} \\ \vdots\\ (m+1)_{b+i+1}\\ \uk\end{pmatrix}$}}=\fS\mbox{\scalebox{0.8}{$\begin{pmatrix}(m)_1\\ \vdots\\ (m)_{b}\\ (m+1)_{b+1}\\ (m)_{b+2}\\ \vdots\\ (m)_{b+i+1}\\ \uk\end{pmatrix}$}}
		+\cdots
		+\fS\mbox{\scalebox{0.8}{$\begin{pmatrix}(m)_1\\ \vdots\\ (m)_{b}\\ (m+1)_{b+1}\\ \vdots\\ (m+1)_{b+i}\\ (m)_{b+i+1}\\ \uk\end{pmatrix}$}}
		+\fS\mbox{\scalebox{0.8}{$\begin{pmatrix}(m)_{1}\\ \vdots\\ (m)_{i-1}\\ (m+1)_{i}\\ \vdots.  \\ (m+1)_{b+i} \\ (m+1)_{b+i+1}\\ \uk\end{pmatrix}$}}.
	\end{align*}
	We point out that the $i=1$ case of the above is simply \eqref{eq:1-2-a-1} with $a\mapsto b+1$. Finally, according to the inductive assumption, we make use of the $a=b$ case of \eqref{eq:S-i-a+1} with $\uk\mapsto \big((m)_{b+i+1},\uk\big)$ to obtain
	\begin{align*}
		\fS\mbox{\scalebox{0.8}{$\begin{pmatrix}(m)_1\\ \vdots\\ (m)_{b}\\ (m+1)_{b+1}\\ (m)_{b+2}\\ \vdots\\ (m)_{b+i+1}\\ \uk\end{pmatrix}$}}
		+\cdots
		+\fS\mbox{\scalebox{0.8}{$\begin{pmatrix}(m)_1\\ \vdots\\ (m)_{b}\\ (m+1)_{b+1}\\ \vdots\\ (m+1)_{b+i}\\ (m)_{b+i+1}\\ \uk\end{pmatrix}$}}=
		\fS\mbox{\scalebox{0.8}{$\begin{pmatrix}(m)_1\\ \vdots\\ (m)_{i-1}\\ (m+1)_{i}\\ (m)_{i+1}\\ \vdots\\ (m)_{b+i+1}\\ \uk\end{pmatrix}$}}
		+\cdots
		+\fS\mbox{\scalebox{0.8}{$\begin{pmatrix}(m)_1\\ \vdots\\ (m)_{i-1}\\ (m+1)_{i}\\ \vdots\\ (m+1)_{b+i}\\ (m)_{b+i+1}\\ \uk\end{pmatrix}$}}.
	\end{align*}
	Combining this with the aligned identity in the above confirms \eqref{eq:S-i-b}.
\end{proof}

We record a corollary of Lemma \ref{le:1-2-a-i}, which plays a central role in the proof of Theorem \ref{th:2a-general}.

\begin{corollary}\label{coro:b-1-important}
	Let $\{F(N)\}_{N\ge 0}$ be a family of series in $q$. For $b\ge 2$,
	\begin{align}
		& \sum_{j=b}^{2b-2} \left(\sum_{N_1,\ldots N_{2b}\ge 0}\frac{q^{\sum_{i=1}^{2b} \binom{N_i+1}{2}+ \sum_{i=b}^{j} (N_i+1) } F(N_{2b})}{(q;q)_{N_1}} \qbinom{N_1}{N_2}\qbinom{N_2}{N_3}\cdots \qbinom{N_{2b-1}}{N_{2b}}\right)\notag\\
		&=\sum_{j=b-1}^{2b-2} \left(\sum_{N_1,\ldots N_{2b}\ge 0}\frac{q^{ \sum_{i=1}^{2b} \binom{N_i+1}{2}+ \sum_{i=b-1}^{j} (N_i+1) }F(N_{2b})}{(q;q)_{N_1}}\qbinom{N_1}{N_2}\qbinom{N_2}{N_3} \cdots \qbinom{N_{2b-1}}{N_{2b}}\right).
	\end{align}
\end{corollary}

\begin{proof}
	In Lemma \ref{le:1-2-a-i} (ii), we take $(m,a,i)\mapsto (1,b-1,b-1)$, and then choose $\uk=\big((1)_{2b-1},(1)_{2b}\big)$. The desired result follows.
\end{proof}

\subsection{Proof of Theorem \ref{th:2a-general}}

Our proof of Theorem \ref{th:2a-general} is based on induction on $a$. For convenience, we write the two sides of \eqref{eq:2a-general} as $\LHS\eqref{eq:2a-general}_a$ and $\RHS\eqref{eq:2a-general}_a$, respectively.

\textit{Base Case.} We first prove the $a=1$ case of \eqref{eq:2a-general}. That is,
\begin{align*}
	\sum_{N_1, N_2\ge 0}\frac{q^{\binom{N_1+1}{2}+\binom{N_2+1}{2}} F(N_2)}{(q;q)_{N_1}}\qbinom{N_1+1}{N_2}=\sum_{N_1, N_2\ge 0}\frac{q^{\binom{N_1+1}{2}+\binom{N_2}{2}} F(N_2)}{(q;q)_{N_1}}\qbinom{N_1}{N_2}.
\end{align*}
Next, we equate the coefficient of $F(N_2)$. Then
$$
\sum_{N_1\ge 0} \frac{q^{N_1(N_1+1)/2+N_2(N_2+1)/2}}{(q;q)_{N_1}} \qbinom{N_1+1}{N_2}=\sum_{N_1 \ge 0} \frac{q^{N_1(N_1+1)/2 +N_2(N_2-1)/2 }}{(q;q)_{N_1}} \qbinom{N_1}{N_2}.
$$
Using \eqref{eq:q-Bin} with $N=\infty$, we can deduce that both sums become
$$
\frac{q^{N_2^2} (-q^{N_2+1};q)_{\infty}}{(q;q)_{N_2}}. 
$$

\textit{Inductive Step.} Assume that \eqref{eq:2a-general} is true for some $a=b-1\ge 1$. We will show \eqref{eq:2a-general} for $a=b$. First, by a similar argument to how we expand $\LHS\eqref{eq:sym-a}$ in the proof of Theorem \ref{th:sym-a}, we have
\begin{align*}
	&\LHS\eqref{eq:2a-general}_{b}\\
	&=\sum_{N_1,\ldots,N_{2b}\ge 0}\frac{q^{\sum_{i=1}^{2b} \binom{N_i+1}{2}} F(N_{2b})}{(q;q)_{N_1}}\qbinom{N_1}{N_2}\cdots \qbinom{N_{2b-1}}{N_{2b}}\\
	&\quad+\sum_{j=b}^{2b-2} \left(
	\sum_{N_1,\ldots, N_{2b}\ge 0}\frac{q^{\sum_{i=1}^{2b} \binom{N_{i}+1}{2} + \sum_{i=b}^{j} (N_i+1) }F(N_{2b})}{(q;q)_{N_1}}\qbinom{N_1}{N_2}\cdots \qbinom{N_{2b-1}}{N_{2b}} \right)\\
	&\quad+\sum_{N_1,\ldots , N_{2b}\ge 0}\frac{q^{\sum_{i=1}^{2b} \binom{N_i+1}{2}+\sum_{i=b}^{2b-1} (N_i+1) }F(N_{2b}+1)}{(q;q)_{N_1}}\qbinom{N_1}{N_2}\cdots \qbinom{N_{2b-1}}{N_{2b}}.
\end{align*}
Let us define an auxiliary series
\begin{align}\label{eq:aux-ser-T}
	\Aux&:=\sum_{N_1,\ldots, N_{2b}\ge 0}\frac{q^{\sum_{i=1}^{2b} \binom{N_i+1}{2}}F(N_{2b})}{(q;q)_{N_1}} \qbinom{N_1}{N_2}\cdots \qbinom{N_{b-2}}{N_{b-1}}  \qbinom{N_{b-1}+1}{N_{b}} \qbinom{N_{b}}{N_{b+1}}\cdots \qbinom{N_{2b-1}}{N_{2b}}.
\end{align}
We may also expand $\Aux$ as
\begin{align*}
	\Aux&=\sum_{N_1,\ldots,N_{2b}\ge 0}\frac{q^{\sum_{i=1}^{2b} \binom{N_i+1}{2}} F(N_{2b})}{(q;q)_{N_1}}\qbinom{N_1}{N_2}\cdots \qbinom{N_{2b-1}}{N_{2b}}\\
	&\quad+\sum_{j=b-1}^{2b-2} \left(\sum_{N_1,\ldots , N_{2b}\ge 0}\frac{q^{ \sum_{i=1}^{2b} \binom{N_1+1}{2}+ \sum_{i=b-1}^{j} (N_i+1) } F(N_{2b})}{(q;q)_{N_1}} \qbinom{N_1}{N_2}\cdots \qbinom{N_{2b-1}}{N_{2b}}\right) \notag\\
	&\quad+\sum_{N_1,\ldots , N_{2b}\ge 0}\frac{ q^{ \sum_{i=1}^{2b}  \binom{N_i+1}{2}+ \sum_{i=b-1}^{2b-1} (N_{i}+1) }F(N_{2b}+1)}{(q;q)_{N_1}}\qbinom{N_1}{N_2}\cdots \qbinom{N_{2b-1}}{N_{2b}}.
\end{align*}
With recourse to Corollary \ref{coro:b-1-important}, we find that
\begin{align}\label{eq:LHS-Aux}
	&\LHS\eqref{eq:2a-general}_{b}-\Aux\notag\\
	&=\sum_{N_1,\ldots,N_{2b}\ge 0}\frac{q^{\sum_{i=1}^{2b} \binom{N_i+1}{2} +\sum_{i=b}^{2b-1} (N_i+1) }F(N_{2b}+1)}{(q;q)_{N_1}}\qbinom{N_1}{N_2}\cdots \qbinom{N_{2b-1}}{N_{2b}}\notag\\
	&\quad-\sum_{N_1,\ldots,N_{2b}\ge 0}\frac{q^{\sum_{i=1}^{2b} \binom{N_i+1}{2}+\sum_{i=b-1}^{2b-1} (N_i+1) }F(N_{2b}+1)}{(q;q)_{N_1}}\qbinom{N_1}{N_2}\cdots \qbinom{N_{2b-1}}{N_{2b}}.
\end{align}
On the other hand, we rewrite $\RHS\eqref{eq:2a-general}_{b}$ as
\begin{align*}
	&\RHS\eqref{eq:2a-general}_{b}\\
	&=\sum_{N_1,\ldots,N_{2b-2}\ge 0}\frac{q^{ \sum_{i=1}^{2b-2} \binom{N_i+1}{2} -\sum_{i=1}^{b-1} N_{2i} }}{(q;q)_{N_1}}\qbinom{N_1}{N_2}\cdots \qbinom{N_{2b-3}}{N_{2b-2}}\\
	&\quad\times \sum_{N_{2b-1},N_{2b}\ge 0}q^{\binom{N_{2b-1}+1}{2}+\binom{N_{2b}}{2}}F(N_{2b})\qbinom{N_{2b-2}}{N_{2b-1}}\qbinom{N_{2b-1}}{N_{2b}}.
\end{align*}
According to the inductive assumption, we make the use of \eqref{eq:2a-general} with $a=b-1$ where we choose
\begin{align*}
	F(N_{2b-2})\mapsto \sum_{N_{2b-1},N_{2b}\ge 0}q^{\binom{N_{2b-1}+1}{2}+\binom{N_{2b}}{2}}F(N_{2b})\qbinom{N_{2b-2}}{N_{2b-1}}\qbinom{N_{2b-1}}{N_{2b}}.
\end{align*}
Therefore,
\begin{align*}
	&\RHS\eqref{eq:2a-general}_{b}\\
	&=\sum_{N_1,\ldots,N_{2b}\ge 0}\frac{q^{ \sum_{i=1}^{2b} \binom{N_i+1}{2} -N_{2b} }F(N_{2b})}{(q;q)_{N_1}}\qbinom{N_1}{N_2}\cdots \qbinom{N_{b-2}}{N_{b-1}}  \qbinom{N_{b-1}+1}{N_{b}} \qbinom{N_{b}}{N_{b+1}}\cdots \qbinom{N_{2b-1}}{N_{2b}}.\\
\end{align*}
Recalling \eqref{eq:aux-ser-T}, we further have
\begin{align*}
	&\RHS\eqref{eq:2a-general}_{b}-\Aux\\
	&=\sum_{N_1,\ldots,N_{2b}\ge 0}\frac{q^{\sum_{i=1}^{2b} \binom{N_i+1}{2}- N_{2b} }F(N_{2b})}{(q;q)_{N_1}}\\
	&\quad\times(1-q^{N_{2b}})\qbinom{N_1}{N_2}\cdots \qbinom{N_{b-2}}{N_{b-1}}  \qbinom{N_{b-1}+1}{N_{b}} \qbinom{N_{b}}{N_{b+1}}\cdots \qbinom{N_{2b-1}}{N_{2b}}\\
	&=\sum_{N_1,\ldots,N_{2b}\ge 0}\frac{ q^{\sum_{i=1}^{2b} \binom{N_i+1}{2}-N_{2b} }F(N_{2b})}{(q;q)_{N_1}}\\
	&\quad\times
	(1-q^{N_{b-1}+1})\qbinom{N_1}{N_2}\cdots \qbinom{N_{b-2}}{N_{b-1}}  \qbinom{N_{b-1}}{N_{b}-1} \qbinom{N_{b}-1}{N_{b+1}-1}\cdots \qbinom{N_{2b-1}-1}{N_{2b}-1}\\
	&=\sum_{N_1,\ldots,N_{2b}\ge 0}\frac{q^{\sum_{i=1}^{2b} \binom{N_i+1}{2}+\sum_{i=b}^{2b-1} (N_i+1) }F(N_{2b}+1)}{(q;q)_{N_1}}\\
	&\quad\times(1-q^{N_{b-1}+1})\qbinom{N_1}{N_2}\cdots \qbinom{N_{b-2}}{N_{b-1}}  \qbinom{N_{b-1}}{N_{b}} \qbinom{N_{b}}{N_{b+1}}\cdots \qbinom{N_{2b-1}}{N_{2b}},
\end{align*}
which tells us that
\begin{align}\label{eq:RHS-Aux}
	&\RHS\eqref{eq:2a-general}_{b}-\Aux\notag\\
	&=\sum_{N_1,\ldots,N_{2b}\ge 0}\frac{q^{ \sum_{i=1}^{2b} \binom{N_i+1}{2}+\sum_{i=b}^{2b-1} (N_i+1)}F(N_{2b}+1)}{(q;q)_{N_1}}\qbinom{N_1}{N_2}\cdots \qbinom{N_{2b-1}}{N_{2b}}\notag\\
	&\quad-\sum_{N_1,\ldots,N_{2b}\ge 0}\frac{q^{\sum_{i=1}^{m} \binom{N_i+1}{2}+ \sum_{i=b-1}^{2b-1} (N_{i}+1) }F(N_{2b}+1)}{(q;q)_{N_1}} \qbinom{N_1}{N_2}\cdots \qbinom{N_{2b-1}}{N_{2b}}.
\end{align}
Combining \eqref{eq:LHS-Aux} and \eqref{eq:RHS-Aux} implies that $\LHS\eqref{eq:2a-general}_{b}=\RHS\eqref{eq:2a-general}_{b}$, which is exactly our desired result. \qed


\section{Concluding Remarks}\label{sec6}

Let us write $W_n=G(2,1,n)$ for the Weyl group of type $B_n$ (or $C_n$). Let $\mathcal{H}^{rat}_{c_1,c_2}(W_n)$ denote the rational Cherednik algebra of $W_n$,  with parameter $c_1$ (resp.~$c_2$) assigned to the reflections associated to hyperplanes $z_i=z_j$ (resp.~$z_i=0$), defined by Etingof--Ginzburg~\cite{EG}. Let $\Omega^{c_1,c_2}(n)$ denote the number of finite dimensional simple modules of $\mathcal{H}^{rat}_{c_1,c_2}(W_n)$. It is known~\cite{GN} (private communication with E. Norton) that 
\begin{equation*} 
\sum_{n\geq 0} \Omega^{\frac{1}{2},1}(n) x^n= \prod_{n\ge 1} \frac{1+x^{2n} }{1-x^{2n} },
\end{equation*}
and
\begin{equation*} 
\quad \sum_{n\geq 0} \Omega^{\frac{1}{2},\frac{1}{2}}(n)  x^n= \prod_{n\ge 1} \frac{1+x^{2n-1}}{1-x^{2n}}.
\end{equation*}

Therefore, our results imply that the number of simple modules in the block of $\mathcal{H}_{-1;-1, 1}(W_n)$ (resp.~$\mathcal{H}_{-1;-1, -1}(W_n)$), labeled by $\omega$, equals the number of finite dimensional simple modules of $\mathcal{H}^{rat}_{\frac{1}{2},1}(W_{n-\omega^2})$ (resp.~$\mathcal{H}^{rat}_{\frac{1}{2},\frac{1}{2}}(W_{n-\omega^2+\omega})$). This provides evidence (in our very special case) for expectations of experts (private communication with P. Shan) that maximal support and minimal support simple modules in dual blocks of category $\mathcal{O}$'s of cyclotomic rational double affine Hecke algebras correspond to each other  under the so-called level-rank duality (see, for example,~\cite{RSVV}).
It is likely that this can be deduced by combining results in~\cite{Gerber} and \cite{RSVV}. 

After we had proved our theorems, we realized that one might derive Theorem~\ref{thm:1.2} and Corollary~\ref{coro:1.4}  by combining results in~\cite{Ariki96,AM,LM} and the Weyl--Kac character formula computations for affine Lie algebras in~\cite{FL}.  However, we emphasize that our approach is more direct and does not rely on deep results from representation theory,  in particular, Ariki's categorification theorem and Weyl--Kac character formula.

It is also shown in \cite{chen} that the infinite product in Theorem~\ref{thm:AM} is an overpartition analogue of the generalized Rogers--Ramanujan identities. Thus, Theorem~\ref{thm:AM} tells us a connection of the multipartitions in $\Lambda^{a,m}(n)$ to the overpartition analogues of the generalized Rogers--Ramanujan type partitions.

\subsection*{Acknowledgements}

S.~Chern was supported by a Killam Postdoctoral Fellowship from the Killam Trusts. T.~Xue was supported in part by the ARC grant DP150103525. T.~Xue thanks Thomas Gerber, Andrew Mathas, Emily Norton and Peng Shan for helpful discussions. A.~J.~Yee was supported in part by a grant ($\#633963$) from the Simons Foundation.

\end{document}